\documentclass[final,onefignum,onetabnum]{siamart220329}

\usepackage[a4paper, left = 1in, right = 1in, bottom = 1in, top = 1in]{geometry}

\usepackage[nocompress]{cite}
\usepackage{lipsum}
\usepackage{amsfonts}
\usepackage{graphicx}
\usepackage{epstopdf}
\usepackage{algorithmic}
\ifpdf
  \DeclareGraphicsExtensions{.eps,.pdf,.png,.jpg}
\else
  \DeclareGraphicsExtensions{.eps}
\fi

\usepackage{subcaption}
\usepackage{dirtytalk}
\usepackage{tikz}
\usetikzlibrary{arrows}
\newsiamremark{remark}{Remark}
\newsiamremark{notation}{Notation}
\newsiamremark{example}{Example}
\newsiamremark{hypothesis}{Hypothesis}
\crefname{hypothesis}{Hypothesis}{Hypotheses}
\newsiamthm{claim}{Claim}

\newcommand{\theoremnum}{}
\newtheorem*{theorem**}{Theorem\theoremnum}
\newenvironment{introtheorem}[1][]{%
  \renewcommand{\theoremnum}{\if\relax\detokenize{#1}\relax\else~#1\fi}
  \begin{theorem**}
}{%
  \end{theorem**}
}

\newcommand{\propositionnum}{}
\newtheorem*{proposition**}{Proposition\propositionnum}
\newenvironment{introprop}[1][]{%
  \renewcommand{\propositionnum}{\if\relax\detokenize{#1}\relax\else~#1\fi}
  \begin{proposition**}
}{%
  \end{proposition**}
}

\mathchardef\mhyphen="2D

\headers{Sparse Signature Coefficient Recovery via Kernels}{D. Shmelev and C. Salvi}

\title{Sparse Signature Coefficient Recovery via
Kernels}

\author{Daniil Shmelev\thanks{Department of Mathematics,
Imperial College London,
  (\email{daniil.shmelev23@imperial.ac.uk}, \email{c.salvi@imperial.ac.uk}).}
\and Cristopher Salvi\footnotemark[1]}

\usepackage{amsopn}

\ifpdf
\hypersetup{
  pdftitle={Sparse Signature Coefficient Recovery via Kernels},
  pdfauthor={D. Shmelev and C. Salvi}
}
\fi

\begin{document}

\maketitle

\begin{abstract}
Central to rough path theory is the signature transform of a path, an infinite series of tensors given by the iterated integrals of the underlying path. The signature poses an effective way to capture sequentially ordered information, thanks both to its rich analytic and algebraic properties as well as its universality when used as a basis to approximate functions on path space. Whilst a truncated version of the signature can be efficiently computed using Chen's identity, there is a lack of efficient methods for computing a sparse collection of iterated integrals contained in high levels of the signature. We address this problem by leveraging signature kernels, defined as the inner product of two signatures, and computable efficiently by means of PDE-based methods. By forming a filter in signature space with which to take kernels, one can effectively isolate specific groups of signature coefficients and, in particular, a singular coefficient at any depth of the transform. We show that such a filter can be expressed as a linear combination of suitable signature transforms and demonstrate empirically the effectiveness of our approach. To conclude, we give an example use case for sparse collections of signature coefficients based on the construction of $N$-step Euler schemes for sparse CDEs.
\end{abstract}



\section{Introduction} \label{chapter:intro}

It is reasonable to assume that, on sufficiently fine time scales, most instances of streamed information—such as text, sound, video, or time series—can be represented as a path $x : [0,1] \to \mathbb R^d$. Chen \cite{chen1957integration} first demonstrated, and subsequent work in \emph{rough path theory} \cite{hambly2010uniqueness, boedihardjo2016signature} further developed and generalized, that any path can be faithfully represented, up to reparameterization, by its collection of iterated integrals, referred to as the \emph{signature}. The latter is defined as the solution in the free tensor algebra $T((\mathbb R^d)) := \prod_{n=0}^\infty (\mathbb R^d)^{\otimes n}$ to the tensor exponential differential equation $dy = y \otimes dx$, serving as the fully non-commutative analogue of the classical exponential function. Running Picard iteration on such differential equation leads to the equivalent, and perhaps more familiar, definition of the signature as collection of iterated integrals $(\int_{0<t_1<...<t_n<1} dx_{t_1} \otimes ... \otimes dx_{t_n})_{n \in \mathbb{N}}$.\par\medskip

It is well known that the restriction of linear functionals on $T((\mathbb R^d))$ to the range of the signature form a unital algebra of real-valued functions that separates points \cite{lyons2004differential}. A straightforward application of the Stone-Weierstrass theorem yields that linear functionals acting on signatures are dense in the space of continuous, real-valued functions on compact sets of unparameterized paths. Thus, the collection of these iterated integrals provide an accurate description of the path, and linear functionals on the signature approximating functions on paths can be easily determined by linear regression, making signature coefficients an ideal set of features for machine learning applications dealing with sequential data \cite{fermanian2023new, cass2024lecture}. Indeed, signature methods have seen a rapid rise in popularity in recent years in the data science community and have been applied in a variety of contexts including deep learning 
\cite{kidger2019deep, morrill2021neural, cirone2023neural, hoglund2023neural, cirone2024theoretical, issa2024non, barancikova2024sigdiffusions}, kernel methods \cite{salvi2021signature, lemercier2021distribution, lemercier2021siggpde, manten2024signature}, quantitative finance \cite{arribas2020sigsdes, salvi2021higher, horvath2023optimal, pannier2024path}, cybersecurity \cite{cochrane2021sk}, and computational neuroscience \cite{holberg2024exact}.\par\medskip

The paths one deals with in data science applications are usually obtained by piecewise linear interpolation of discrete time series. Leveraging Chen's relation and the elementary property that the signature of a linear path equals the tensor exponential of its increment, the signature of a piecewise linear path $x = x^1 * ... *x^L$ can be elegantly computed using the expression
\begin{equation}\label{eqn:sig_pwl}
    S(x) = \exp(\Delta x^1)\otimes ... \otimes \exp(\Delta x^L),
\end{equation}
which is used by all the main python packages for computing signatures such as \texttt{esig} \cite{esig}, \texttt{iisignature} \cite{reizenstein2018iisignature} and \texttt{signatory} \cite{signatory}. The expression in (\ref{eqn:sig_pwl}) provides a dense computation of the signature with complexity $\mathcal{O}(Ld^n)$, where $n \in \mathbb{N}$ is the truncation level, which quickly becomes impractical to implement due to the exponential growth in the number of signature features. However, in many cases, only a small number of signature coefficients are necessary during inference, allowing the majority to be ignored, for example after applying regularization techniques such as Lasso. With Chen's relation, an isolated signature coefficient at level $n$ of the transform can be computed at a complexity of $\mathcal{O}(Ln^2)$, with parallelisation allowing this to be reduced to $\mathcal{O}(Ln)$. Yet beyond this, there is currently no available numerical technique that enables the computation of only a sparse subset of signature coefficients. \par\medskip

In this paper, we introduce a novel technique for recovering sparse sets of signature coefficients by computing the inner product of the full signature with an appropriately designed filter. When the filter is derived from the signatures of carefully selected paths, the resulting inner product can be efficiently computed as a linear combination of signature kernels. These kernels can, in turn, be evaluated using PDE-based kernel methods \cite{salvi2021signature}. Our method will have a serial complexity of $\mathcal{O}(Ln 2^n)$, but with the important benefit that the dependence on $n$ can be parallelised away, resulting in an algorithm of complexity $\mathcal{O}(L)$ after parallelization.\par\medskip

Our main results can be summarised as follows. Given a continuous path $x \in C_1([0,1], \mathbb{R}^n)$, let $\lambda \odot x$ denote the component-wise scaling of an $n$-dimensional path $x$ by $\lambda = (\lambda_1, \ldots, \lambda_n) \in \mathbb{R}^n$, such that the resulting path has components $\lambda \odot x = (\lambda_1 x^{(1)}, \ldots, \lambda_n x^{(n)})$. We note that the signature of the resulting path takes the form
\begin{equation*}
    S(\lambda \odot x)^J = S(x)^{J} \prod_{k=1}^m \lambda_{j_k}.
\end{equation*}

for $J = (j_1, \ldots, j_m) \in \{1,\ldots, n\}^m$. For a multi-index $I$, let $\mathcal{P}(I)$ denote the set of multi-indices which are permutations of $I$. Differentiating with respect to the above scaling of the underlying path, we get
\begin{equation*}
    \frac{\partial^n}{\partial \lambda_1 \cdots \partial \lambda_n} S(\lambda \odot x)^J \hspace{0.5mm} \biggr\rvert_{\lambda = 0} \hspace{1mm} = \begin{cases}
        S(x)^J, &J \in \mathcal{P}(1,\ldots, n),\\
        0, &J \notin \mathcal{P}(1,\ldots, n).
        \end{cases}
\end{equation*}

That is, we are left only with coefficients whose multi-index is a permutation of $(1,\ldots, n)$, which we will call the \textit{anagram class} of $(1,\ldots, n)$. It follows that taking kernels with the linear path $y$, whose signature is given by
\begin{equation*}
    S(y)^J = |J|!
\end{equation*}
extracts the sum $S(x)^{\mathcal{P}(1,\ldots,n)}$, weighted by $n!$. This forms our first result, \textit{anagram class isolation}. For two paths $x,y$ of the same dimension, let $\mathbf{k}_{x,y} = \left< S(x), S(y) \right>$ denote the corresponding \textit{signature kernel}. Given a $d$-dimensional path $x \in C_1([0,1], \mathbb{R}^d)$ with components $x = (x^{(1)}, \ldots, x^{(d)})$, and a multi-index $I = (i_1, \ldots, i_n) \in \{1,\ldots, d\}^n$, let $x^I \in C_1([0,1], \mathbb{R}^n)$ denote the reordered path whose components are given by $x^I = (x^{(i_1)}, \ldots, x^{(i_n)})$.

\begin{introprop}[\ref{thm:kderiv}] \text{\normalfont(Anagram Class Isolation)}
Let $x \in C_1([0,1], \mathbb{R}^d)$ and let $y \in C_1([0,1], \mathbb{R}^n)$ be the linear path $y_t = (t,\ldots, t)$. Then
    \begin{equation*}
        \frac{\partial^n}{\partial \lambda_1 \cdots \partial \lambda_n} \mathbf{k}_{\hspace{0.5mm} x^I, \hspace{0.5mm} \lambda \odot y} \hspace{0.5mm} \biggr\rvert_{\lambda = 0} \hspace{1mm} = \hspace{1mm} \frac{1}{n!} \hspace{1mm} S(x)^{\mathcal{P}(I)}.
    \end{equation*}
\end{introprop}

If we intend to recover specific coefficients, we should select the path with which to take kernels more carefully. We look for a \say{filtering} path $z$, such that the action of taking kernels with $z$ isolates a specific permutation, or order, of $I$ in $\mathcal{P}(I)$. The next result shows that this is possible.

\begin{introprop}[\ref{prop:order_isolation_exact}] \text{\normalfont(Order Isolation)}
Let $x \in C_1([0,1], \mathbb{R}^d)$. Then there exists a path $z \in C_1([0,1], \mathbb{R}^n)$ such that
    \begin{equation*}
        \frac{\partial^n}{\partial \lambda_1 \cdots \partial \lambda_n} \mathbf{k}_{\hspace{0.5mm} x^I, \hspace{0.5mm} \lambda \odot z} \hspace{0.5mm} \biggr\rvert_{\lambda = 0} \hspace{1mm} = \hspace{1mm} S(x)^I.
    \end{equation*}
\end{introprop}

By using a finite difference $D^{(n)}_\lambda \cdot \big\rvert_{\lambda = 0}$ to approximate the derivative $\partial^n / \partial \lambda_1 \cdots \partial \lambda_n$, we can obtain an approximation for the coefficient $S(x)^I$. If $D^{(n)}_\lambda$ is chosen to be exact on linear functions of $\lambda_1, \ldots, \lambda_n$, such as a forward or central difference, we will note that the error in the approximation originates entirely from levels of the signature deeper than the $n^{th}$. By finding a suitable way of zeroing these levels, we arrive at our main result - an algorithm for recovering a signature coefficient from repeated kernel evaluations.

\begin{introtheorem}[\ref{thm:main}]
    There exists a path $z \in C_1([0,1], \mathbb{R}^n)$ and scalars $\alpha_i = \alpha_i(M) \in \mathbb{R}$, $\beta_i = \beta_i(M) \in \mathbb{R}$ such that for any $x \in C_1([0,1], \mathbb{R}^d)$,
    \begin{equation*}
        \sum_{i=0}^M \alpha_i D^{(n)}_\lambda \mathbf{k}_{x^I, \hspace{0.5mm} \beta_i \lambda \odot z} \Big\rvert_{\lambda = 0} \to S(x)^I
    \end{equation*}
    
    as $M \to \infty$.
\end{introtheorem}

We will show that this convergence is quick, and as such small values of $M$ are sufficient in practice. With suitable parallelization, the above sum can be computed efficiently in $\mathcal{O}(L)$ time by leveraging the result that the signature kernel $\mathbf{k}$ satisfies a Goursat PDE \cite{salvi2021signature}. We believe this remarkable result will have a significant impact in the data science community as it is, to our knowledge, the first allowing to compute coefficients in level $n$ of the signature with a complexity that is independent of $n$. With the advent of massively parallel computing, this poses a feasible alternative to Chen's relation when computational speed is of importance.\par\medskip

The paper is organized as follows. We begin in Section \ref{chapter:intro} by introducing fundamental concepts and definitions, discussing direct approaches to signature computation and outlining our proposed methodology. In Section \ref{chapter:anagram}, we discuss how, given a target multi-index, one can filter out coefficients up to permutation of the multi-index by considering derivatives of path scalings. We will refer to this group of coefficients as the \textit{anagram class} of the target coefficient. Having achieved this, in Section \ref{chapter:order} we will offer a way to extract a specific coefficient from the anagram class, by taking kernels with the signature of an axis path. We discuss a numerical scheme for computing these kernels in Section \ref{chapter:pde}. An important benefit of using Chen's relation to compute isolated coefficients is the ability to immediately recover related coefficients from intermediary steps in the computation. In Section \ref{section:grid_retrieval}, we note that this can also be achieved with the algorithm we propose. Moreover, with certain variations of our algorithm, one can recover a significantly larger proportion of coefficients from intermediary steps than with Chen's relation, without the need for additional computations. Finally, we discuss a potential application of our approach in Section \ref{chapter:example}, and present concluding remarks in Section \ref{chapter:conclusion}.\par\medskip

A basic implementation of our proposed method along with code reproducing the results of this paper is available publicly at \url{https://github.com/daniil-shmelev/sigcoeff}.

\subsection{Computation of the Signature Transform} \label{section:signature_transform}
We begin by giving a brief introduction to some fundamental definitions and results, and outlining the basic approach to computing signatures. Throughout, we will take $(V, \lVert \cdot \rVert_V)$ to be a finite-dimensional Banach space over $\mathbb{R}$, equipped with an inner product $\left<\cdot, \cdot\right>_V$ and an orthonormal basis $\{e_i : 1 \leq i \leq d\}$ with corresponding dual basis $\{e^*_i : 1 \leq i \leq d\}$. From Section \ref{chapter:anagram} onwards, $V$ will always be taken to be the Euclidean space $\mathbb{R}^d$.

\begin{definition}
    Denote by $C_p([a,b], V)$ the space of continuous paths from $[a,b]$ to $V$ of finite p-variation, written $C_p(V)$ when the interval can be inferred from context.
\end{definition}

\begin{definition}
    Define $T(V) = \bigotimes_{i=0}^{\infty} V^{\otimes i}$ to be the tensor algebra of formal polynomials over $V$ endowed with the usual operations of $+$ and $\otimes$.
\end{definition}

\begin{definition} [Signature Transform]
    Let $x \in C_p([a,b], V)$ for $p \in [1,2)$. Then for any $[s,t] \subseteq [a,b]$, define the $k^{th}$ \textit{level} of the signature transform as the iterated Young integral
    \begin{equation*}
    S(x)_{[s,t]}^{(k)} = \int_{s < t_1 < \cdots < t_k < t} dx_{t_1} \otimes dx_{t_2} \otimes \cdots \otimes dx_{t_k} \in V^{\otimes k}
    \end{equation*}
    
    and define the signature transform of $x$ on $[s,t]$ to be the formal series

    \begin{equation*}
        S(x)_{[s,t]} = \left(1, S(x)_{[s,t]}^{(1)}, \ldots, S(x)_{[s,t]}^{(k)}, \ldots \right) \in T((V)).
    \end{equation*}

    We will often drop the subscript $[s,t]$ when it is clear from context.
    
\end{definition}

\begin{definition} [Signature Coefficient]
    Let $I = (i_1, \ldots, i_k)$ be a multi-index of integers in $\{1,\ldots, d\}$. Let $e^*_I := e^*_{i_1} \cdots e^*_{i_k} \in (V^*)^{\otimes k}$. By viewing $e^*_I$ as an element of $T((V))^*$, define the scalar
    \begin{equation*}
        S(x)^I = e^*_I(S(x)) \in \mathbb{R},
    \end{equation*}
    which we may write as the iterated integral
    \begin{equation}\label{defn:signature_coeff}
        \int_{s < t_1 < \cdots < t_k < t} dx_{t_1}^{(i_1)} dx_{t_2}^{(i_2)} \cdots dx_{t_k}^{(i_k)},
    \end{equation}
    where $x^{(i)}$ denotes the $i^{th}$ \textit{channel} of the path $x_t = (x^{(1)}, \ldots, x^{(d)})_t$.
\end{definition}

An important property of the signature transform is that its levels decay factorially, motivating the use of only the first $n$-many levels in a truncated version of the transform. Chen's relation for the signature of concatenated paths, combined with the simplicity of the signature of a linear path, allows for an efficient computation of this truncated signature.

\begin{lemma}[Factorial Decay {\cite[Lemma 5.1]{lyons2014rough}}]\label{thm:factorialdecay}
    Let $x \in C_1([a,b], V)$ and $k \in \mathbb{N}$. Then
    \begin{equation*}
        \left\lVert S(x)^{(k)} \right\rVert_{V^{\otimes k}} \leq \frac{\lVert x \rVert ^k _{1,[a,b]}}{k!},
    \end{equation*}
    where $\lVert x \rVert _{1,[a,b]}$ is the 1-variation of $x$ on $[a,b]$. 
\end{lemma}

\begin{proposition}[{\cite[Section 1.3.1]{cass2024lecture}}] \label{prop:linearsignature}
    Let $x \in C_1([a,b], V)$ be the linear path
    \begin{equation*}
        x_t = x_a + \frac{t - a}{b-a}(x_b - x_a).
    \end{equation*}
    Let $x_{a,b}$ denote the canonical inclusion of $x_b - x_a \in V$ into $T((V))$. Then
    \begin{equation*}
        S(x)_{[a,b]} = \exp(x_{a,b}),
    \end{equation*}
    where $\exp$ denotes the tensor exponential
    \begin{equation*}
        \exp(v) := \sum_{k=0}^\infty \frac{v^{\otimes k}}{k!}.
    \end{equation*}
\end{proposition}

\begin{proposition} [Chen's relation { \cite{chen1954iterated}}] \label{prop:chen} Let $x \in C_p([a,b], V)$ and $y \in C_p([b,c], V)$ for $p \in [1,2)$. Then
\begin{equation*}
    S(x * y)_{[a,c]} = S(x)_{[a,b]} \otimes S(y)_{[b,c]},
\end{equation*}
where $x*y$ denotes the path concatenation of $x$ and $y$.
\end{proposition}

Suppose we have a path, $x \in C_1([0,1], V)$, given by the linear interpolation of $L+1$ many points, such that
\begin{equation*}
    x = \hat{x}_{1} * \hat{x}_{2} * \cdots * \hat{x}_{L}
\end{equation*}
where $\hat{x}_k$ are linear paths. For the avoidance of doubt, we will refer to $\left\lVert x \right\rVert_1$ as the \textit{$1$-variation} of the path, and to $L$ as the \textit{path length}. Given a multi-index $I = (i_1, \ldots, i_n)$, suppose we are interested in computing the signature coefficient $S(x)^I$. A simple approach using Chen's relation is to iteratively compute
\begin{equation} \label{eq:coefficient_chen}
    S(\hat{x}_{1} * \cdots * \hat{x}_{j+1})^{(i_1, \ldots, i_m)} = \sum_{k=1}^m S(\hat{x}_{1} * \cdots * \hat{x}_{j})^{(i_1, \ldots, i_k)} \cdot S(\hat{x}_{j+1})^{(i_{k+1}, \ldots, i_m)}
\end{equation}
for all $1 \leq m \leq k$, where $S(\hat{x}_{j+1})^{(i_{k+1}, \ldots, i_m)}$ is known from Proposition \ref{prop:linearsignature}. Such a computation has complexity $\mathcal{O}(Ln^2)$, but can be reduced down to $\mathcal{O}(Ln)$ by computing the $n$ coefficients required at each step in parallel. As we will see in Section \ref{chapter:example}, in practical use cases for sparse collections of signature coefficients, it is often the case that if $S(x)_{[0,1]}^I$ is required, then so is $S(x)_{[0,t]}^J$ for some $J \subset I$, $t \leq 1$. It is easy to see that when $J = (i_1, \ldots, i_m)$ is an initial segment of $I$, and $t$ is the end point of some linear segment $\hat{x}_k$, one can recover $S(x)^J_{[0,t]}$ from intermediate steps in the above approach without the need for additional computations.

\subsection{The Signature Kernel}
We now define the inner product of two signatures, called the \textit{signature kernel}. As discussed in the introduction, we will see later how these kernels can be used to \say{filter out} the coefficients which we are interested in from the full signature.

\begin{definition} [Hilbert-Schmidt inner product]
Given an inner product $\left<\cdot, \cdot \right>_V$ on $V$, equip $V^{\otimes k}$ with the inner product
\begin{equation*}
    \left<u, v \right>_{V^{\otimes k}} = \prod_{i=1}^k \left< u_i, v_i \right>_V
\end{equation*}
for any $u = u_1 \cdots u_k$, $v = v_1 \cdots v_k \in V^{\otimes k}$ and define an inner product on $T(V)$ by
\begin{equation*}
    \left< A,B \right>_{T(V)} = \sum_{k=0}^{\infty} \left<A_k, B_k \right>_{V ^ {\otimes k}}
\end{equation*}
for any $A = (A_0, A_1, \ldots)$, $B = (B_0, B_1, \ldots) \in T(V)$. Define the \textit{extended tensor algebra} $T((V))$ to be the completion of $T(V)$ with respect to $\left< \cdot, \cdot \right>_{T(V)}$.
\end{definition}

\begin{definition}[Signature Kernel]
    Let $x \in C_1([a,b], V)$ and $y \in C_1([c,d], V)$. The signature kernel $\mathbf{k}_{x,y} : [a,b] \times [c,d] \rightarrow \mathbb{R}$ is given by 
    \begin{equation*}
        \mathbf{k}_{x,y}(t,s) = \left< S(x)_{[a,t]}, S(y)_{[c,s]} \right>_{T((V))}.
    \end{equation*}

    We will often write $\mathbf{k}_{x,y}$ when the kernel is taken over the entire interval.
\end{definition}

By applying the Cauchy-Schwarz inequality, we immediately get the following corollary of Lemma \ref{thm:factorialdecay}.

\begin{corollary} \label{cor:cauchyfactorialdecay}
    Let $x, y \in C_1([a,b], V)$ and $k \in \mathbb{N}$. Then
    \begin{equation*}
        \left\lvert \left< S(x)_{[a,b]}^{(k)}, S(y)_{[c,d]}^{(k)} \right>_{V^{\otimes k}} \right\rvert \leq \frac{\lVert x \rVert ^k _{1,[a,b]} \lVert y \rVert ^k _{1,[c,d]}}{(k!)^2}.
    \end{equation*}
\end{corollary}

Occasionally, it will be helpful to refer to a truncated signature kernel in which we only consider the signature up to the $n^{th}$ level. We define this as follows.

\begin{definition}[Truncated Signature Kernel]
    Let $x \in C_1([a,b], V)$ and $y \in C_1([c,d], V)$. The truncated signature kernel $\mathbf{k}^n_{x,y} : [a,b] \times [c,d] \rightarrow \mathbb{R}$ is given by
    \begin{equation*}
        \mathbf{k}^n_{x,y}(t,s) = \sum_{k=0}^n \left< S(x)^{(k)}_{[a,t]}, S(y)^{(k)}_{[c,s]} \right>_{V^{\otimes k}}.
    \end{equation*}
\end{definition}

\begin{theorem}[Signature Kernel PDE {\cite[Theorem 2.5]{salvi2021signature}}]\label{thm:goursatpde}
    Let $x \in C_1([a,b], V)$ and $y \in C_1([c,d], V)$. Then $\mathbf{k}_{x,y}$ is the solution of the Goursat PDE
    \begin{equation}\label{sigkernelpde}
        \frac{\partial ^ 2 \mathbf{k}_{x,y}}{\partial t \partial s} = \left< \dot{x}_t, \dot{y}_s \right>_V \mathbf{k}_{x,y}, \quad \mathbf{k}_{x,y}(a, \cdot) = \mathbf{k}_{x,y}(\cdot, c) = 1,
    \end{equation}
    the existence and uniqueness of a solution to which follows from \cite[\textit{Theorems 2 \& 4}]{lees1960goursat}.
\end{theorem}

In \cite{salvi2021signature} it was shown that for paths obtained by linearly interpolating discrete input streams of length $L$ and dimension $d$, computation of the signature kernel can be achieved in $\mathcal{O}(L^2d)$ time. The advantage of computing the full signature kernel over the truncated version is that numerical PDE schemes lend themselves well to parallelisation. As such, the complexity may be reduced to $\mathcal{O}(Ld)$ on suitable GPU hardware by parallelising along anti-diagonals of the discretisation grid (see Section \ref{chapter:pde}).

\subsection{Universal Nonlinearity}
Arguably one of the most important results in the study of the signature transform is that of its \textit{univeral nonlinearity}, which states that continuous functions on path space can be well approximated by linear functionals on the signature.

\begin{theorem} [Universal Nonlinearity {\cite[Theorem 3.3]{lyons2022signature}}] \label{thm:signature_universality}
    Given a suitable topology on $\mathcal{C}_1([a,b], V)$, let $\mathcal{K} \subset \mathcal{C}_1([a,b], V)$ be compact and $f : \mathcal{K} \rightarrow \mathbb{R}$ a continuous function. Then for any $\varepsilon > 0$ there exists a truncation level $n \in \mathbb{N}$ and $\alpha_{i,I} \in \mathbb{R}$ such that for all $x \in \mathcal{K}$
    \begin{equation*}
        \left \lvert f(x) - \sum_{i=0}^n \sum_{I \in \{1, \ldots, d\}^i} \alpha_{i,I} \hspace{0.5mm} S(x)_{[a,b]}^I \right \rvert \leq \varepsilon.
    \end{equation*}
\end{theorem}

This is extended to $\mathcal{C}_p$ in \cite[Theorem 1.4.7]{cass2024lecture}. A discussion of suitable topologies on $\mathcal{C}_p$ can be found in \cite{cass2024lecture, cass2024topologies}. Theorem \ref{thm:signature_universality} provides our main motivation for computing isolated signature coefficients. Suppose for a given function $f$ and tolerance $\varepsilon$ we know the coefficients $\alpha_{i,I}$. Moreover, suppose the $\alpha_{i,I}$ are sparse and mostly vanish. Let $x$ be a piecewise linear path obtained from interpolating an input data stream of length $L$. To compute $f(x)$ one can, of course, compute $S(x)$ up to level $n$ using Chen's relation to get an approximate value for $f(x)$. However, for a path of length $L$ this would have a computational complexity of $\mathcal{O}(Ld^n)$ \cite{kidger2020signatory, reizenstein2018iisignature} and would involve unnecessary computation of unused signature coefficients. As discussed earlier, a more careful application of Chen's relation using Equation \eqref{eq:coefficient_chen} reduces the complexity significantly to $\mathcal{O}(Ln^2)$, or $\mathcal{O}(Ln)$ after parallelisation of the computation. The approach we present in this paper will have a serial complexity of $\mathcal{O}(Ln 2^n)$, but with the benefit of being easily parallelisable down to $\mathcal{O}(L)$, surpassing Chen's relation by a factor of $n$.

\subsection{Methodology}

Given a multi-index $I = (i_1, \ldots, i_n) \in \{1,\ldots, d\}^n$ representing the signature coefficient which we aim to compute, the main idea of the approach is to create a filter $F \in \text{span}(\mathcal{S})$ such that for any multi-index $J$,
\begin{equation*}
F^{J} =
\begin{cases}
1 &\text{ if } J = I,\\
0 &\text{ otherwise.}
\end{cases}
\end{equation*}

We then note that, for $x \in C_1(V)$, $\left< S(x), F \right> = S(x)^{I}$. If $F$ can be expressed linearly in terms of signature transforms of paths, then we may rewrite the above inner product as a linear combination of signature kernels, which by Theorem \ref{thm:goursatpde} are computable efficiently as solutions of Goursat PDEs. We introduce some definitions which will aid with our construction of $F$.

\begin{definition}
    Given a multi-index $J$, let $\mathcal{P}(J)$ denote the set of multi-indices which are permutations of $J$.
\end{definition}

\begin{example}
    $\mathcal{P}(1,2,2) = \{(1,2,2), (2,1,2), (2,2,1)\}$.
\end{example}

\begin{definition}
    For a set $\mathcal{I}$ of multi-indices, let $S(x)^{\mathcal{I}}$ denote the sum $\sum_{I \in \mathcal{I}} S(x)^I$.
\end{definition}

Within our construction of $F$, we will extensively make use of component-wise path scalings, applied to channels of $x$ which are given by the multi-index $I$. We will define these path scalings, as well as the path $x$ restricted to $I$, as follows.

\begin{definition}
    Let $\lambda = (\lambda_1, \ldots, \lambda_d) \in \mathbb{R}^d$. For $z \in C_p(V)$ given by $z_t = \sum_{i=1}^d z^{(i)}_{t} e_i$, denote by $\lambda \odot z \in C_1(V)$ the path given by $(\lambda \odot z)_t = \sum_{i=1}^d \lambda_i z^{(i)}_t e_i$.
\end{definition}

\begin{definition}
    For given a path $x = \left(x^{(1)}, \ldots, x^{(d)}\right) \in C_1([0,1], \mathbb{R}^d)$ and multi-index $I = (i_1, \ldots, i_n)$, we will denote by $x^I \in C_1([0,1], \mathbb{R}^n)$ the path with components $x^I = \left(x^{(i_1)}, \ldots, x^{(i_n)}\right)$.
\end{definition}

The problem of forming $F$ using signature transforms can be split into 2 sub-problems:

\begin{enumerate}
    \item Anagram class isolation: How do we zero all coefficients given by multi-indices outside of $\mathcal{P}(I)$?
    \item Order isolation: How do we zero coefficients given by multi-indices within $\mathcal{P}(I)$ whilst setting the coefficient at index $I$ to $1$?
\end{enumerate}

\section{Anagram Class Isolation} \label{chapter:anagram}

We solve the problem of anagram class isolation by applying a path scaling $\lambda = (\lambda_1, \ldots, \lambda_n) \in \mathbb{R}^n$ to the path $x^I$ and noting that, on the level of signature coefficients, the resulting multiplicative factor is constant up to reordering of the multi-index. For the anagram class $\mathcal{P}(I)$, this factor is precisely the product $\lambda_1 \cdots \lambda_n$ of the components of the scaling. Moreover, this factor is the unique monomial which does not vanish after differentiating with respect to $\lambda_1, \ldots, \lambda_n$ at 0. We formalise this in the following proposition.

\begin{proposition} \label{thm:kderiv}
Let $x \in C_1([0,1], \mathbb{R}^d)$ and let $y \in C_1([0,1], \mathbb{R}^n)$ be the linear path $y_t = t \mathbf{e}$, where $\mathbf{e} = e_1 +  \cdots + e_n$. Then
    \begin{equation} \label{eq:kderiv}
        \frac{\partial^n}{\partial \lambda_1 \cdots \partial \lambda_n} \mathbf{k}_{\hspace{0.5mm} x^I, \hspace{0.5mm} \lambda \odot y} \hspace{0.5mm} \biggr\rvert_{\lambda = 0} \hspace{1mm} = \hspace{1mm} \frac{1}{n!} \hspace{1mm} S(x)^{\mathcal{P}(I)}.
    \end{equation}
\end{proposition}

The proof follows from the fact that, for any path $z$ and multi-index $J = (j_1, \ldots, j_m)$,
\begin{equation*}
    S(\lambda \odot z)^{J} = S(z)^{J} \prod_{k=1}^m \lambda_{j_k}.
\end{equation*}
The full proof is detailed in Appendix \ref{proof:kderiv}.

\begin{remark}
    Let $E$ be the Banach space $E = \{v \in T((V)) : \lVert v \rVert < \infty\}$, and $S^{-1}(E) \subset C_1(V)$ be the pre-image of $E$ under the signature transform acting on $C_1(V)$. Then by considering the restriction of $S$ to $S^{-1}(E)$, we can restate the above result without relying on deterministic scalings by considering instead the Gateaux derivative
    \begin{equation*}
        D^n S(0)\{x^{(i_1)} e_1, \ldots, x^{(i_n)} e_n\},
    \end{equation*}
    where $x^{(i)} e_i$ is understood to be the path $(x^{(i)}_t e_i)_t \in C_1([0,1], V)$. Intuitively, we see that the only coefficients remaining after evaluating such a derivative are those of first order in $x^{(i_1)}, \ldots, x^{(i_n)}$, that is, precisely those given by multi-indices in $\mathcal{P}(I)$.
\end{remark}

\begin{notation}
    Throughout, given a function $f : \mathbb{R}^n \to \mathbb{R}$, we will denote by $D^{(n)}_\lambda \cdot \big\rvert_{\lambda = 0}$ any finite difference operator of the form
    \begin{equation} \label{eq:D_defn}
        D^{(n)}_\lambda f(\lambda) \Big\rvert_{\lambda = 0} = \sum_{\widetilde \lambda \in \Lambda} C_{\widetilde \lambda} f(\widetilde \lambda)
    \end{equation}
    approximating the derivative $\partial^n / \partial \lambda_1 \cdots \partial\lambda_n |_{\lambda = 0}$, such that $D^{(n)}_\lambda$ is exact on linear functions of $(\lambda_1, \ldots, \lambda_n)$. In particular, for some step-size $h>0$, $D^{(n)}_\lambda$ may be the forward difference
    \begin{equation} \label{eq:FD_defn}
        D^{(n)}_\lambda f(\lambda) \Big\rvert_{\lambda = 0} = FD^{(n)}_\lambda f(\lambda) \Big\rvert_{\lambda = 0} := \frac{1}{h^n} \sum_{\widetilde \lambda \in\{0,h\}} (-1)^{\text{sgn}(\widetilde \lambda)} f(\widetilde \lambda),
    \end{equation}
     where $\text{sgn}(\widetilde \lambda) = \sum_i \mathbf{1}\{\widetilde \lambda_i \neq h\}$, or the central difference
    \begin{equation} \label{eq:CD_defn}
        D^{(n)}_\lambda f(\lambda) \Big\rvert_{\lambda = 0} = CD^{(n)}_\lambda f(\lambda) \Big\rvert_{\lambda = 0} := \frac{1}{(2h)^n} \sum_{\widetilde \lambda \in\{-h,h\}} (-1)^{\text{sgn}(\widetilde \lambda)} f(\widetilde \lambda).
    \end{equation}
\end{notation}

As a consequence of Proposition \ref{thm:kderiv}, we obtain the finite difference approximation
\begin{equation} \label{eq:finite_diff}
    n! D^{(n)}_\lambda \mathbf{k}_{x^I, \lambda \odot y} \Big\rvert_{\lambda = 0} \approx S(x)^{\mathcal{P}(I)}.
\end{equation}

Since the signature kernel of $n$-dimensional paths can be computed as the solution of a PDE in $\mathcal{O}(Ln^2)$ time, the approximation in Equation \eqref{eq:finite_diff} has a serial computational complexity of $\mathcal{O}(Ln^2 2^n)$ for a $2^n$-term finite difference operator such as the forward or central difference. However, computing cross-derivatives of high order is a notoriously difficult numerical task and often unstable as the step size tends to zero. In the sections that follow, we present two approaches to mitigate this instability: an exact approach using truncated signature kernels, and an untruncated approximation benefiting from simpler computation.

\subsection{A Truncated Kernel Approach}

It is easy to see that the approximation error in Equation \ref{eq:finite_diff} is accumulated from levels of the signature deeper than the $n^{th}$. As a result, it can be avoided completely by considering truncated kernels instead.

\begin{proposition} \label{prop:truncatedkernel}
    Let $x \in C_1([0,1], \mathbb{R}^d)$ and let $y \in C_1([0,1], \mathbb{R}^n)$ be the linear path $y_t = t \mathbf{e}$. Then
    \begin{equation}
            n! D^{(n)}_\lambda \mathbf{k}^n_{x^{I}, \hspace{1mm} \lambda \odot y} \Big\rvert_{\lambda = 0} = S(x)^{\mathcal{P}(I)},
    \end{equation}
    
    where $\mathbf{k}^n$ is the truncated signature kernel up to level $n$.
\end{proposition}

\begin{proof}
    The result follows immediately from Proposition \ref{thm:kderiv} and the definition of $D^{(n)}_\lambda$.
\end{proof}

In particular, we may fix the step size of the finite difference to be large in order to restrict numerical instabilities associated with small values of $h$. Efficient approaches for computing truncated signature kernels for a linearly interpolated data stream of length $L$ are presented in \cite{kiraly2019kernels}, although these are typically non-linear in $L$ or reliant on low-rank approximations. In light of this, we may prefer to approximate the result of Proposition \ref{prop:truncatedkernel} using untruncated kernels, by finding a different way of zeroing higher levels of the signature.

\subsection{Vandermonde Systems of Kernels}

For convenience, given $\lambda \in \mathbb{R}^d$ and $\beta \in \mathbb{R}^+$ we define the $\lambda \mhyphen \beta$ signature kernel as
\begin{equation*}
    \mathbf{k}^{\lambda \mhyphen \beta}_{x,y}(s,t) = \mathbf{k}_{x, \hspace{1mm} \beta \lambda \odot y}(s,t) = \mathbf{k}_{\beta \lambda \odot x, \hspace{1mm} y}(s,t),
\end{equation*}

and the corresponding truncated kernel as
\begin{equation*}
    \mathbf{k}^{\lambda \mhyphen \beta, n}_{x,y}(s,t) = \mathbf{k}^n_{x, \hspace{1mm} \beta \lambda \odot y}(s,t) = \mathbf{k}^n_{\beta \lambda \odot x, \hspace{1mm} y}(s,t).
\end{equation*}

Proposition \ref{prop:truncatedkernel} can then be rewritten as

\begin{equation} \label{eq:vandermonde_sum_omega_phi}
        n! D^{(n)}_\lambda \mathbf{k}^{\lambda \mhyphen 1, n}_{x^I,y} \Big\rvert_{\lambda = 0} = S(x)^{\mathcal{P}(I)}.
\end{equation}

Whilst this is an exact form for $S(x)^{\mathcal{P}(I)}$ in terms of truncated kernels, one might prefer to consider untruncated kernels as these are easier to compute as PDE solutions. If we could zero the first few levels of the signature after the $n^{th}$, we may then rely on the factorial decay of signature coefficients to minimize the error from higher levels. To do this, consider a linear combination of kernels
\begin{equation*}
\sum_{i=0}^M \alpha_i \mathbf{k}^{\lambda \mhyphen \beta_i}_{x^I,y}
\end{equation*}

for some $\alpha_i, \beta_i \in \mathbb{R}$, where $M \geq 0$ is a parameter which we will refer to as the \textit{depth} of the scaling. Suppose we wish to solve for $\alpha_i, \beta_i$ such that the resulting sum of kernels preserves the $n^{th}$ level of the signature, whilst zeroing levels $n+1, \ldots, n+M$. For levels deeper than $n+M$, we rely on the factorial decay in Lemma \ref{thm:factorialdecay} to keep the error low. Then we require that
\begin{equation*}
    \sum_{i=0}^M \alpha_i \beta_i^m = \delta_{m,n}, \quad \forall n\leq m \leq n+M.
\end{equation*}
\par\medskip

We may rewrite the above equations as the Vandermonde matrix equation
\begin{equation} \label{eq:vandermonde}
B_{n,M} \cdot \alpha = 
\begin{pmatrix}
\beta_0^n & \beta_1^n & \cdots & \beta_M^n\\
\beta_0^{n+1} & \beta_1^{n+1} & \cdots & \beta_M^{n+1}\\
\vdots & \vdots & \ddots & \vdots\\
\beta_0^{n+M} & \beta_1^{n+M} & \cdots & \beta_M^{n+M}\\
\end{pmatrix}
\begin{pmatrix}
\alpha_0 \\ \alpha_1 \\ \vdots \\ \alpha_M
\end{pmatrix}
=
\begin{pmatrix}
1 \\ 0 \\ \vdots \\ 0
\end{pmatrix}.
\end{equation}
\par\medskip

Such generalized Vandermonde matrices $B_{n,M}$ are invertible for distinct $\beta_i>0$. An explicit inversion is given by the following.

\begin{proposition} \label{prop:vandermonde_inversion}
    For distinct $\beta_i > 0$, an explicit solution to Equation \eqref{eq:vandermonde} is given by
    \begin{equation*}
        \alpha_i = \frac{(-1)^M}{\beta_i^n} \prod_{\substack{j = 0 \\ j \neq i}}^M \frac{\beta_j}{\beta_i - \beta_j}.
    \end{equation*}
\end{proposition}

\begin{proof}
    See \cite{arafat2022fast} and Appendix \ref{proof:vandermonde_inversion}.
\end{proof}

Note that if any of the $\beta_i$ are chosen greater than $1$, there will be signature coefficients in levels beyond the $(n+M)^{th}$ which are scaled by a high power of $\beta_i$, which may cause high error. To exclude this possible source of error, the $\beta_i$ should be chosen in $(0,1]$ to ensure that $\beta_i^m$ does not grow as $m \to \infty$. Beyond this, the method is not particularly sensitive to the choice of $\beta_i$ as long as these are chosen reasonably, as per the conditions of Proposition \ref{prop:vandermonde} below. Combining this scaling with Equation \eqref{eq:finite_diff}, we get:

\begin{proposition} \label{prop:vandermonde}
    Fix $n\geq 1$ and let $\beta_i = \beta_i(M) \in (0,1]$ be chosen such that for any positive constant $C > 0$,
    \begin{equation} \label{eq:alphacond}
        \max_{0 \leq i \leq M} |\alpha_i| = \max_{0 \leq i \leq M} \frac{1}{\beta_i^n} \prod_{\substack{j = 0 \\ j \neq i}}^M \frac{\beta_j}{|\beta_i - \beta_j|} = \mathcal{O}\left(\frac{[(n+M)!]^2}{M} C^{M} \right)
    \end{equation}
    as $M \to \infty$. Let $y \in C_1([0,1], \mathbb{R}^n)$ be the linear path $y_t = t \mathbf{e}$. Then for any $x \in C_1([0,1], \mathbb{R}^d)$,
    \begin{equation} \label{eq:vandermonde_sum}
        n! \sum_{i=0}^M \alpha_i D^{(n)}_\lambda \mathbf{k}^{\lambda \mhyphen \beta_i}_{x^I, y} \Big\rvert_{\lambda = 0} \to S(x)^{\mathcal{P}(I)}
    \end{equation}
    
    as $M \to \infty$.
\end{proposition}

\begin{proof}
    See Appendix \ref{proof:vandermonde}.
\end{proof}

\begin{example}
    The uniform choice $\beta_i = (i+1)/(M+1)$ satisfies the above condition since
    \begin{align*}
        \max_{0\leq i \leq M} |\alpha_i| &= \max_{0\leq i \leq M} \frac{1}{\beta_i^n} \prod_{\substack{j=0 \\ j\neq i}}^M \frac{\beta_j}{|\beta_i - \beta_j|} \\
        &= \max_{0\leq i \leq M} \left(\frac{M+1}{i+1}\right)^n \prod_{\substack{j=0 \\ j\neq i}}^M \frac{j+1}{|i-j|} \\
        &= \max_{0\leq i \leq M} \left(\frac{M+1}{i+1}\right)^{n+1} \genfrac{(}{)}{0pt}{}{M}{i}\\
        &= \mathcal{O}(M^{n+1} 2^M).
    \end{align*}
\end{example}

In practice, when $n$ is large we can choose the $\beta_i$ such that their $n^{th}$ powers are uniform to minimize numerical errors associated with overly strong path scalings and a blow-up of the $1/\beta_i^n$ factor in $\alpha_i$. That is, we choose $\beta_i = [(i+1)/(M+1)]^{1/n}$. This can be shown to satisfy Condition \eqref{eq:alphacond} for all $n$. The convergence in $M$ is quick, as suggested by the following proposition.\par\medskip

\begin{proposition}\label{prop:M_bound}
    Let $\beta_i = [(i+1)/(M+1)]^{1/n}$. For a fixed $n \geq 1$, there exists constants $A > 0$ dependent on $\lVert x^I \rVert_1$ and $n$, and $B > 0$ dependent on $n$, such that
    \begin{equation} \label{eq:M_bound}
        \left\lVert n! \sum_{i=0}^M \alpha_i D^{(n)}_\lambda \mathbf{k}^{\lambda \mhyphen \beta_i}_{x^I, y} \Big\rvert_{\lambda = 0} - S(x)^{\mathcal{P}(I)}\right \rVert \leq A (M+1)^2 B^M \frac{ \left\lVert x^I \right \rVert_1^{n+M}}{((n+M)!)^2}.
    \end{equation}
\end{proposition}

\begin{proof}
    See Appendix \ref{proof:M_bound}.
\end{proof}

\begin{remark}
    When evaluating sums such as \eqref{eq:vandermonde_sum} numerically, the error can be reduced by subtracting $\mathit{1}$ from every kernel, thus effectively only considering level $\mathit{1}$ and above of the signature. It is easy to see that this change does not affect the result, but may significantly reduce the magnitude of each term of the sum and stabilise the scheme.
\end{remark}

The bound \eqref{eq:M_bound} successfully captures the error pattern produced by the Vandermonde approximation in practice, as we will see later (see Figures \ref{fig:depth_error} and \ref{fig:M_error}). Specifically, as $n$ increases, extracting the signature becomes progressively more challenging, requiring more kernel evaluations, increasing the error. At the same time, factorial decay of coefficients ensures that the residual levels of the signature which are not zeroed by the Vandermonde approximation are less significant. This results in a humped structure of the absolute error with respect to the depth of the coefficient, $n$. In addition, the bound captures the quick decay of error with respect to $M$. In practice, a choice of $M = 1$ or $2$ is usually sufficient. \par\medskip

For a $2^n$-term finite-difference operator such as forward or central difference, the resulting algorithm involves $(M+1)2^n$ kernel evaluations, giving it a computational complexity of $\mathcal{O}(LnM2^n)$. Indeed, due to the $1/((n+M)!)^2$ factor in \eqref{eq:M_bound}, $M$ is only particularly relevant for small values of $n$ and can be set to $0$ for $n$ sufficiently large. As such, we do not include it in any reasoning about computational complexity from now on.

\section{Order Isolation} \label{chapter:order}

If we replace the term $\mathbf{k}_{\hspace{0.5mm} x^I, \hspace{0.5mm} \lambda \odot y}$ in Proposition \ref{thm:kderiv} by $S(\lambda \odot x^I)$, we can view the derivative component-wise as
\begin{equation*}
    \frac{\partial^n}{\partial \lambda_1 \cdots \partial \lambda_n} S(\lambda \odot x^I)^J \hspace{0.5mm} \biggr\rvert_{\lambda = 0} \hspace{1mm} = \begin{cases}
        S(x^I)^J, &J \in \mathcal{P}(1,\ldots, n),\\
        0, &J \notin \mathcal{P}(1,\ldots, n).
        \end{cases}
\end{equation*}

Taking kernels with the linear path $y$, whose signature is given by
\begin{equation*}
    S(y)^J = |J|!
\end{equation*}
extracts the sum $S(x)^{\mathcal{P}(I)}$, weighted by $n!$. Similarly, if we replace the term $\mathbf{k}_{x^I, y}^{\lambda \mhyphen \beta_i}$ in Proposition \ref{prop:vandermonde} by $S(\beta_i \lambda \odot x^I)$, we recover the component-wise limit
\begin{equation}
    \sum_{i=0}^M \alpha_i D^{(n)}_\lambda S(\beta_i \lambda \odot x^I) \Big\rvert_{\lambda = 0} \to Z \in T((\mathbb{R}^n))
\end{equation}
where $Z$ is such that for all multi-indices $J = (j_1, \ldots, j_m)$ with $m \leq n$, $Z^J = 1$ if $J \in \mathcal{P}(1,\ldots, n)$ and $Z^J = 0$ otherwise. If we wish to extract just the coefficient $S(x)^I$, we must therefore replace $y$ with a path $z$, whose signature satisfies
\begin{equation} \label{eq:z_cond}
    S(z)^J = \begin{cases}
        1, &J = (1,\ldots, n),\\
        0, &J \in \mathcal{P}(1,\ldots, n), J \neq (1,\ldots, n).
    \end{cases}
\end{equation}
It is simple to see, using Chen's relation, that the axis path given by $z_t = (e_1 * e_2 * \cdots * e_n)_t$ exactly satisfies Condition \eqref{eq:z_cond}. This leads us to the following adaptation of Proposition \ref{thm:kderiv}, as well as our main result, Theorem \ref{thm:main}.

\begin{proposition} \label{prop:order_isolation_exact}
Let $x \in C_1([0,1], \mathbb{R}^d)$ and let $z$ denote the axis path given by $z_t = (e_1 * e_2 * \cdots * e_n)_t$ for $t \in [0,1]$, where $(e_i)_t$ is understood to be the linear path from $\mathbf{0}$ to $e_i$. Then
    \begin{equation}
        \frac{\partial^n}{\partial \lambda_1 \cdots \partial \lambda_n} \mathbf{k}_{\hspace{0.5mm} x^I, \hspace{0.5mm} \lambda \odot z} \hspace{0.5mm} \biggr\rvert_{\lambda = 0} \hspace{1mm} = \hspace{1mm} S(x)^I.
    \end{equation}
\end{proposition}

\begin{proof}
    The proof follows exactly that of Proposition \ref{prop:vandermonde} by using $z$ in place of $y$ and noting that $z$ satisfies Condition \eqref{eq:z_cond}.
\end{proof}

\begin{theorem}\label{thm:main}
    Let $\beta_i$ be chosen to satisfy Condition \ref{eq:alphacond} and $\alpha_i$ be as in Proposition \ref{prop:vandermonde_inversion}. Let $z$ denote the axis path given by $z_t = (e_1 * e_2 * \cdots * e_n)_t$ for $t \in [0,1]$, where $(e_i)_t$ is understood to be the linear path from $\mathbf{0}$ to $e_i$. Then for any $x \in C_1([0,1], \mathbb{R}^d)$,
    \begin{equation*}
        \sum_{i=0}^M \alpha_i D^{(n)}_\lambda \mathbf{k}^{\lambda \mhyphen \beta_i}_{x^I, z} \Big\rvert_{\lambda = 0} \to S(x)^I
    \end{equation*}
    
    as $M \to \infty$.
\end{theorem}
\begin{proof}
    The proof follows exactly that of Proposition \ref{prop:vandermonde} by using $z$ in place of $y$ and noting that $z$ satisfies Condition \eqref{eq:z_cond}.
\end{proof}

\begin{figure*}[b!]
    \centering
    \begin{subfigure}[t]{0.49\textwidth}
        \centering
        \includegraphics[width = \textwidth]{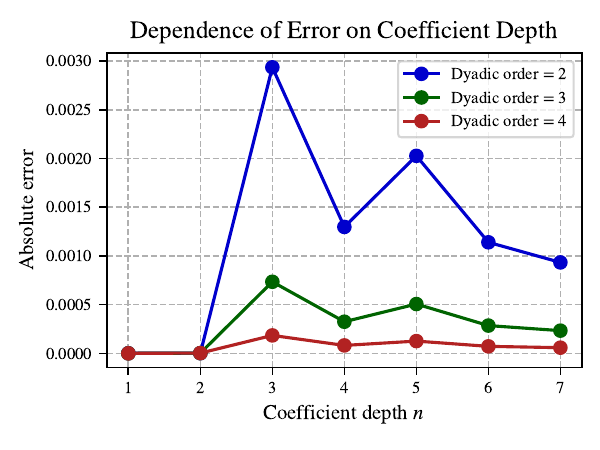}
    \end{subfigure}%
    ~ 
    \begin{subfigure}[t]{0.49\textwidth}
        \centering
        \includegraphics[width = \textwidth]{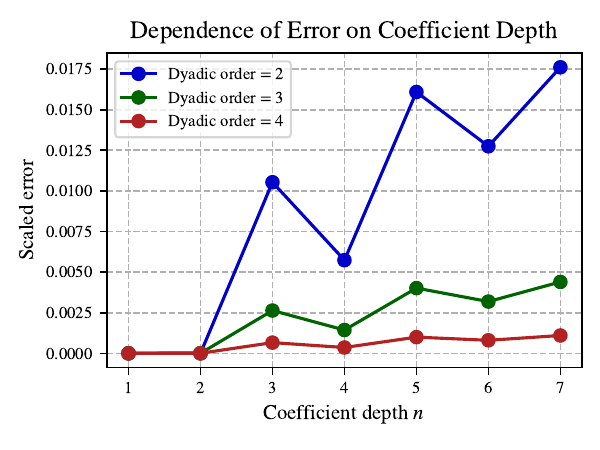}
    \end{subfigure}
    \caption{Average errors when computing $S(x)^{(1,\ldots,n)}$ over 1,000 random paths constrained to $[0,1]^d$, with path length $L = 150$ and scaling depth $M = 3$. We compute average absolute errors (left) and average absolute errors scaled by average coefficient magnitude (right).}
    \label{fig:depth_error}
\end{figure*}

\begin{remark}
    Returning to the original formulation of our approach, the filter $F$ with which we effectively take inner products can now be written as
    \begin{equation*}
        F = \sum_{i=0}^M \alpha_i D^{(n)}_\lambda S(\beta_i \lambda \odot z) \Big\rvert_{\lambda = 0} \in \text{span}(\mathcal{S}).
    \end{equation*}
    Theorem \ref{thm:main} then states that $\left<S(x^I), F\right> \to S(x)^{I}$ as $M \to \infty$. As remarked in the previous section, this convergence is quick.
\end{remark}

A natural question is whether there exist other choices of filtering path $z$ which immediately extract linear combinations of signature coefficients with multi-indices in $\mathcal{P}(I)$. A simple example of this involves computing sums of signature coefficients with \say{blocks} of unordered indices. Suppose, for example, we wish to compute
\begin{equation*}
    S(x)^{(1,2,3)} + S(x)^{(2,1,3)},
\end{equation*}

that is, the sum of coefficients given by a multi-index where 1 and 2 appear before 3, but the order of 1 and 2 is not important. Instead of considering the path $z = e_1 * e_2 * e_3$, we consider $z = (e_1 + e_2) * e_3$ to remove the order constraint on 1 and 2. It is easy to see by Chen's relation that
\begin{align*}
S(z)^{(1,2,3)} = S(z)^{(2,1,3)} &= 1/2,\\
S(z)^{(1,3,2)} = S(z)^{(2,3,1)} = S(z)^{(3,1,2)} = S(z)^{(3,2,1)} &= 0,
\end{align*}

giving the desired isolation pattern on $\mathcal{P}(1,2,3)$. To generalise this idea, we introduce the \textit{concatenation} of multi-indices. For two multi-indices $I = (i_1, \ldots, i_{m_1})$ and $J = (j_1, \ldots, j_{m_2})$, denote the concatenation of $I$ and $J$ by $I*J = (i_1, \ldots, i_{m_1}, j_1, \ldots, j_{m_2})$. For two sets of multi-indices $\mathcal{I}$ and $\mathcal{J}$, denote by $\mathcal{I}*\mathcal{J}$ the set
\begin{equation*}
    \mathcal{I}*\mathcal{J} = \{I * J : I \in \mathcal{I}, J \in \mathcal{J}\}.
\end{equation*}

\begin{figure*}[b!]
    \centering
    \includegraphics[width = 0.6\textwidth, trim={1cm 0.2cm 1cm 0.2cm},clip]{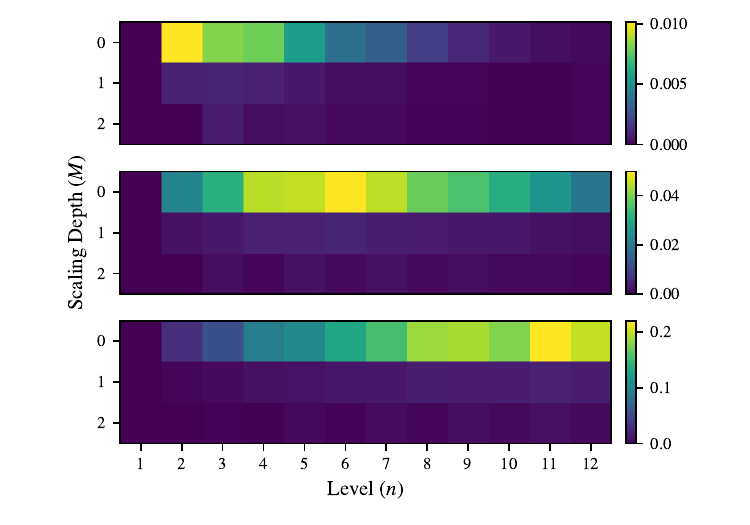}
    \caption{Average errors when computing $S(x)^{(1,2,\ldots, n)}$ for random paths in $[0,1]^d$, with $L = 100$ (top), $L = 500$ (middle) and $L = 1000$ (bottom). Dyadic order of the PDE solver is set to 3.}
    \label{fig:M_error}
\end{figure*}

The following result follows immediately from Chen's relation and Proposition \ref{prop:vandermonde}:

\begin{theorem} \label{thm:semiordered}
    Let $\beta_i$ be chosen to satisfy Condition \ref{eq:alphacond} and let $\alpha_i$ be as in Proposition \ref{prop:vandermonde_inversion}. Suppose $I_1, \ldots, I_m$ are multi-indices of lengths $l_1, \ldots, l_m$ respectively, such that $I_1 * \cdots * I_m = I$. Let $j_i = \sum_{p=1}^i l_p$ for $i = 1, \ldots, m$ and denote by $z$ the path
    \begin{equation*}
        z_t = ((e_1 + \cdots + e_{j_1}) * (e_{j_1+1} + \cdots + e_{j_2}) * \cdots * (e_{j_{m-1}+1} + \cdots + e_n))_t
    \end{equation*}
    for $t \in [0,1]$, where $(e_i)_t$ is understood to be the linear path from $\mathbf{0}$ to $e_i$. Then
    \begin{equation*}
    \sum_{i=0}^M \alpha_i D^{(n)}_\lambda \mathbf{k}^{\lambda \mhyphen \beta_i}_{x^I,z} \Big\rvert_{\lambda = 0} \to \frac{1}{l_1! l_2! \cdots l_m!} \hspace{1mm} S(x)^{\mathcal{P}(I_1) * \mathcal{P}(I_2) * \cdots * \mathcal{P}(I_m)}
    \end{equation*}
    as $M \to \infty$, for any $x \in C_1([0,1], \mathbb{R}^d)$.
\end{theorem}

\begin{proof}
    See Appendix \ref{proof:semiordered}.
\end{proof}

\begin{remark}
    The single coefficient case corresponds to $I_j = (j)$, since then $\mathcal{P}(I_1) * \mathcal{P}(I_2) * \cdots * \mathcal{P}(I_m) = I$. Similarly, isolation of $S(x)^{\mathcal{P}(I)}$ corresponds to $I_1 = I$.
\end{remark}

Theorem \ref{thm:semiordered} has a simple interpretation. Let
\begin{align*}
    \tilde{x}^{i}_t &:= \int_{[0,t]^{l_i}} \prod_{j \in I_i} dx^{(j)}_{u_j}\\
    &= \prod_{j \in I_i} \left(x^{(j)}_t - x^{(j)}_0\right)
\end{align*}

and define $\tilde{x} = (\tilde{x}^{1}, \ldots, \tilde{x}^{m})$, where $\tilde{x}$ can either be viewed as a path of $l_i$-dimensional integrals with respect to the channels of $x$ given by indices in $I_i$, or equivalently as a product of channels as above. Then we may write
\begin{equation*}
S(x)^{\mathcal{P}(I_1) * \mathcal{P}(I_2) * \cdots * \mathcal{P}(I_m)} = S(\tilde{x})^{(1,\ldots,m)}.
\end{equation*}

Whether there is a simple way to construct a filtering path $z$ which extracts an arbitrary linear combination of coefficients in $\mathcal{P}(I)$ is beyond the scope of this paper, and left for future research. It is reasonable to suggest that an adapted signature inversion technique, such as those presented in \cite{chang2019insertion, lyons2017hyperbolic, lyons2018inverting}, could be used to solve a general form of the constraint \eqref{eq:z_cond}.

\section{The Kernel Goursat PDE} \label{chapter:pde}

Recall the signature kernel PDE \eqref{sigkernelpde} is given by
\begin{equation} \label{eq:general_PDE}
    \frac{\partial ^ 2 \mathbf{k}_{x,y}}{\partial t \partial s} = \left< \dot{x}_t, \dot{y}_s \right>_V \mathbf{k}_{x,y}, \hspace{5mm} \mathbf{k}_{x,y}(0, \cdot) = \mathbf{k}_{x,y}(\cdot, 0) = 1
\end{equation}

for general $x,y \in C_1([0,1], V)$. In the case where $y$ is the axis path $z$ as in Theorem \ref{thm:main}, this becomes

\begin{equation} \label{eq:special_PDE}
    \frac{\partial ^ 2 \mathbf{k}_{x,z}}{\partial t \partial s} = \dot{x}_t^{(m)} \mathbf{k}_{x,z}, \quad \text{for } s \in \left[\frac{m}{d}, \frac{m+1}{d}\right], \quad \mathbf{k}_{x,z}(0, \cdot) = \mathbf{k}_{x,z}(\cdot, 0) = 1
\end{equation}

assuming $z$ attains its $i^{th}$ node at $t = i/d$. Taking inspiration from \cite{day1966runge, salvi2021signature, wazwaz1993numerical}, we propose the finite difference scheme
\begin{align} \label{eq:kernelFinDiff}
    \widehat{\mathbf{k}}(t_{i+1}, s_{j+1}) &= \left(\widehat{\mathbf{k}}(t_{i+1}, s_j) + \widehat{\mathbf{k}}(t_i, s_{j+1})\right) A(t_i, s_j) - \widehat{\mathbf{k}}(t_i, s_j) B(t_i, s_j)\\
    A(t_i, s_j) &= \left(1 + \frac{1}{2} \Delta_j x_{t_i}  + \frac{1}{12}\Delta_j x_{t_i} ^2\right) \nonumber \\
    B(t_i, s_j) &= \left(1 - \frac{1}{12} \Delta_j x_{t_i} ^2\right) \nonumber \\
    \Delta_j x_{t_i} &= (x^{(\hat j)}_{t_{i+1}} - x^{(\hat j)}_{t_i}) / 2^{\gamma_2}, \quad \hat j = \lfloor j/2^{\gamma_2} \rfloor\nonumber
\end{align}

over the dyadically refined grid $P_{\gamma_1, \gamma_2} = \{(t_i, s_j) : 0 \leq i \leq 2^{\gamma_1} L, \hspace{1mm} 0\leq j \leq 2^{\gamma_2} n\}$ of order $(\gamma_1, \gamma_2)$, where $L$ is the length of the discrete data stream $x$.\par\medskip

As in \cite{salvi2021signature}, for $n$-dimensional paths, the corresponding scheme for the PDE \eqref{eq:general_PDE} carries a serial complexity of $\mathcal{O}(Ln^2)$, since one must traverse a PDE grid of size $\mathcal{O}(Ln)$, evaluating $\left< \dot{x}_t, \dot{y}_s \right>$ at each step in $\mathcal{O}(n)$ time. If we traverse the PDE grid diagonal-wise rather than row-wise or column-wise, we note that each diagonal depends only on entries in the previous one, and so we may parallelise the computation within each diagonal. Assuming $L \geq n$, this reduces the complexity to $\mathcal{O}(Ln)$. In the case of the PDE \eqref{eq:special_PDE}, we no longer need to compute $\left< \dot{x}_t, \dot{y}_s \right>$ at each step of the computation, resulting in a serial complexity of $\mathcal{O}(Ln)$. The same parallelisation trick reduces this to $\mathcal{O}(L)$.\par\medskip

Note that the sum in Theorem \ref{thm:main} consists of $\mathcal{O}(2^n)$ many kernels, each satisfying a PDE of the form \eqref{eq:special_PDE}, leading to a serial complexity of $\mathcal{O}(Ln2^n)$. By applying the above trick to parallelise each individual kernel evaluation, and noting that each term of the sum can be computed independently, we see that the sum can be evaluated at a parallel complexity of $\mathcal{O}(L)$.

\begin{theorem}[{\cite[Theorem 3.5]{salvi2021signature}}]
    Suppose there exists a constant $M$, independent of $\gamma_1, \gamma_2$ such that
    \begin{equation*}
        \sup_{s,t} \lvert \left< \dot{x}_t, \dot{z}_s \right> \rvert < M.
    \end{equation*}
    Then there exists a constant $K > 0$, depending on $M$ and $\mathbf{k}_{x,y}$ but independent of $\gamma_1, \gamma_2$, such that
    \begin{equation*}
        \sup_{s,t} \left\lvert \mathbf{k}_{x,z}(t,s) - \widehat{\mathbf{k}}_{x,z}(t,s) \right\rvert \leq \frac{K}{2^{\gamma_1 + \gamma_2}}, \quad \forall \gamma_1, \gamma_2 \geq 0.
    \end{equation*}
\end{theorem}

\begin{remark}
    Given that usually $L \gg n$, one may choose to differ the dyadic refinement over $s$ and $t$ and take $\gamma_1 \neq \gamma_2$. For simplicity, we will take $\gamma_1 = \gamma_2 = \gamma$ in our experiments.
\end{remark}

\begin{figure*}[t!]
    \centering
    \begin{subfigure}[t]{0.49\textwidth}
        \centering
        \includegraphics[width = \textwidth]{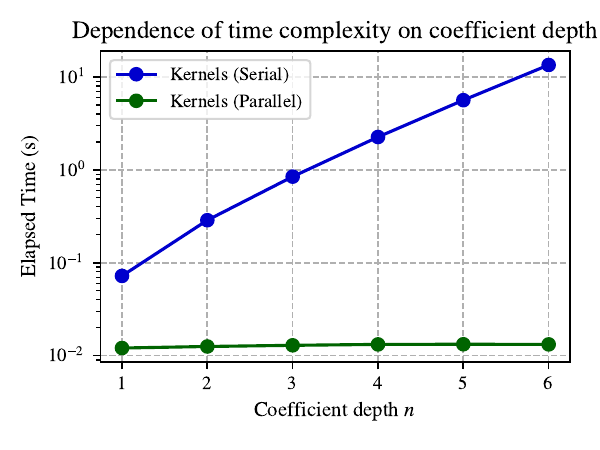}
    \end{subfigure}%
    ~ 
    \begin{subfigure}[t]{0.49\textwidth}
        \centering
        \includegraphics[width = \textwidth]{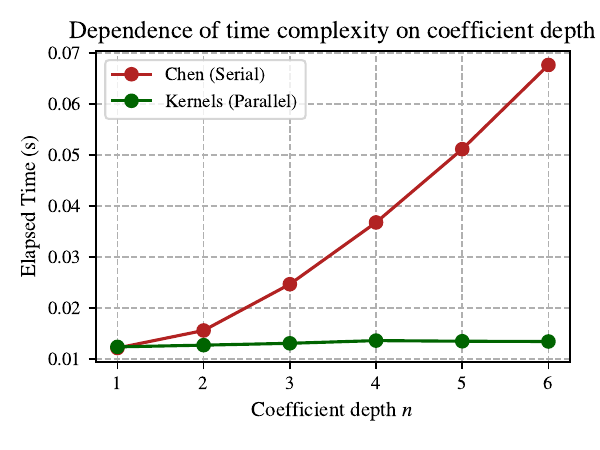}
    \end{subfigure}
    \caption{Dependence of runtime on coefficient depth, using serial versus parallel kernel computations (left, log scale), and  parallel kernel computations versus a serial computation using Chen's relation (right, linear scale). Runtime is averaged over 10 random paths constrained to $[0,1]^d$, with path length $L = 10,000$, scaling depth $M = 1$ and PDE dyadic order set to $1$. For comparability, all computations are performed on an NVIDIA Tesla P100 GPU.}
    \label{fig:complexity}
\end{figure*}

\subsection{Batch Retrieval of Coefficients and Choice of \texorpdfstring{$D^{(n)}_\lambda$}{TEXT}} \label{section:grid_retrieval}

As we will see in Section \ref{chapter:example}, in practical use cases for sparse collections of signature coefficients, it is often the case that if $S(x)^I$ is required, then so is $S(x)^J$ for some $J \subset I$. As we saw in Section \ref{section:signature_transform}, one of the benefits of using Chen's relation is that the coefficient $S(x)_{[0,t]}^J$ can be extracted from intermediate steps in the computation of $S(x)_{[0,1]}^I$ without the need for additional computations, assuming $J$ is an initial segment of $I$ and $t \leq 1$. We note in this section that the same can be achieved with Theorem 3.1. For convenience, we will refer to $S(x)_{[0,t]}^J$ as a \say{sub-coefficient} of $S(x)_{[0,1]}^I$ if $J \subset I$ and $t \leq 1$, where $t$ is assumed to be a node in the refined grid $P_{\gamma_1, \gamma_2}$. \par \medskip

In the discussion that follows, we will denote by $z_k$ the axis path
\begin{equation*}
    (z_k)_t := (e_1 * e_2 * \cdots * e_k)_t
\end{equation*}
for $t \in [0,1]$, where $(e_i)_t$ is understood to be the linear path from $\mathbf{0}$ to $e_i$. Suppose we have computed $S(x)^I$ using Theorem \ref{thm:main} and the scheme \eqref{eq:kernelFinDiff}, choosing to use the central difference $D_\lambda^{(n)} = CD_\lambda^{(n)}$ with a step-size $h=1$, and storing all intermediate values along the PDE grid 
\begin{equation*}
P^{|I|}_{\gamma_1, \gamma_2} := \{(t_j, s_k) : 0 \leq j \leq 2^{\gamma_1} L, \hspace{1mm} 0\leq k \leq 2^{\gamma_2} |I|\}.
\end{equation*}
Then one obtains the set
\begin{equation*}
    \mathfrak{K}_I := \left\{\mathbf{k}^{\widetilde \lambda \mhyphen \beta_i}_{x^I, z_{|I|}}(t_j, s_k) : \widetilde \lambda \in \Lambda_{CD}^{(n)}, 0 \leq i \leq M, (t_j, s_k) \in P^{|I|}_{\gamma_1, \gamma_2}\right\}
\end{equation*}

where $\Lambda_{CD}^{(n)}$ is the set of arguments required for the central difference $CD^{(n)}_\lambda$ as in \eqref{eq:D_defn} and \eqref{eq:CD_defn}. For $I = (i_1, \ldots, i_n)$, let $J = (i_1, \ldots, i_m)$ where $m<n$. Then it is easy to see that
\begin{equation*}
    \mathfrak{K}_J \subset \mathfrak{K}_I,
\end{equation*}
and so one can recover $S(x)^J$ without additional kernel computations. The same argument applies if one wishes to recover $S(x)_{[0,t]}^J$ having computed $S(x)^I_{[0,1]}$ for $J \subset I$, $t \leq 1$. If instead of initial segments one is more interested in sub-coefficients given by final segments $J = (i_m, \ldots, i_n)$, one can perform the same trick but with the time-reversed paths $\overleftarrow{x}, \overleftarrow{z}$. \par\medskip

With certain choices of $D^{(n)}_\lambda$, one can go even further. Our recommended choice for $D^{(n)}_\lambda$ in general is the central difference $CD^{(n)}_\lambda$ with step-size $h = 1$, due to its numerical stability and accuracy. However, suppose we choose the forward difference $FD^{(n)}_\lambda$, and we are interested in computing $S(x)^I$ and $S(x)^J$ where $J$ is a subset of $I$ but not necessarily an initial (or final) segment. Let $\lambda^{I,J} \in \mathbb{R}^n$ be such that $\lambda^{I,J}_k = \mathbf{1}\{i_k \in J\}$. Then it is easy to see that if the step-size is taken as $h = 1$, 
\begin{equation*}
    FD^{(m)}_\lambda \mathbf{k}^{\lambda \mhyphen \beta_i}_{x^J, z_m} \Big\rvert_{\lambda = 0} = FD^{(n)}_\lambda \mathbf{k}^{(\lambda^{I,J} \odot \lambda) \mhyphen \beta_i}_{x^I, z_n} \Big\rvert_{\lambda = 0},
\end{equation*}
which allows for the recovery of $S(x)^J$, since $\lambda^{I,J} \odot \lambda \in \Lambda_{FD}^{(n)}$ for all $\lambda \in \{0,1\}$. Thus, whilst the central difference allows for recovery of sub-coefficients whose multi-indices are initial (or final) segments of $I$, the forward difference allows for recovery of \textit{all} sub-coefficients.

\section{Example: State-Space Models with Time-Dependent Flows} \label{chapter:example}
In this Section, we present an example use case of signatures where only a small number of coefficients of the transform are needed, and computation of the full signature is unnecessary. As discussed in  \cite[Section 7.1]{friz2010multidimensional} and \cite[Section 3.2]{cass2024lecture}, the $N^{th}$ order truncated signature can be used to form $N$-step Euler discretisation schemes for CDEs driven by the path $x$.

\begin{definition}
    Given a vector field $g = (g_1, \ldots, g_d)^T : \mathbb{R}^m \to \mathbb{R}^m$, recall its identification with the first order differential operator
    \begin{equation*}
        \sum_{i=1}^m g_i(y) \frac{\partial}{\partial y^i}.
    \end{equation*}
    Given a family of smooth vector fields $f = (f_1, \ldots, f_d)$ on $\mathbb{R}^m$, let $\Phi_f$ denote the algebra homomorphism from $T((\mathbb{R}^d))$ to the space of differential operators, given by
    \begin{align*}
        \Phi_f(1) &= I\\
        \Phi_f(e_{i_1} \cdots e_{i_k}) &= f_{i_1} \cdots f_{i_k}
    \end{align*}
    where $I$ is the identity operator and $f_{i_1} \cdots f_{i_k}$ denotes a composition of differential operators.
\end{definition}

Given $x_t \in \mathbb{R}^d$ and $y_0 \in \mathbb{R}^m$, consider the CDE
\begin{equation} \label{eq:ode}
    dy_t = f(y_t) dx_t = \sum_{i=1}^d f_i(y_t) dx^{(i)}_t.
\end{equation}

A Taylor expansion leads to the $N$-step Euler scheme

\begin{equation} \label{eq:EulerFull}
    y_t \approx y_s + \sum_{\substack{i_1, \ldots, i_k \in \{1,\ldots, d\} \\ 1 \leq k \leq N}} f_{i_1} \cdots f_{i_k} I(y_s) S(x)^{(i_1, \ldots, i_k)}_{[s,t]},
\end{equation}

for $0 < t - s \ll 1 $, which for convenience can be written more compactly as
\begin{equation} \label{eq:EulerCompact}
    y_t \approx \mathcal{E}_{s,t}^N (y_s; f, x)
\end{equation}

where $\mathcal{E}_{s,t}^N (y_s; f, x) := \Phi_f \left( S^N(x)_{[s,t]} \right) I(y_s)$. Repeated iterations of the scheme over a fine partition $\{0 = t_0 < t_1 < \cdots < t_n = T\}$ lead to an approximate solution for $y_t$ over a desired interval $[0,T]$. We consider some examples of the CDE \eqref{eq:ode} where the resulting scheme is sparse and depends on a small number of coefficients, giving a potential use case for Theorem \ref{thm:main}. Specifically, we will consider sparsely connected finite state-space models, which restrict the non-zero coefficients in \eqref{eq:EulerCompact}. The following definition will significantly simplify later arguments.

\begin{definition} \label{defn:linear_cde}
    Consider a CDE of the form
    \begin{equation}\label{eq:linear_cde}
        dy_t = \sum_{i=1}^d A^i y_t dx^{(i)}_t
    \end{equation}
    for some matrices $A^i \in \mathbb{R}^{m \times m}$. We define the directional graph $G=(V,E)$ corresponding to the CDE by
    \begin{align*}
        V &= \{y^{(1)}, \ldots, y^{(m)}\},\\
        E &= \bigcup_{i=1}^d \{(y^{(j)}, y^{(k)}) : A^i_{j,k} \neq 0\}.
    \end{align*}
    In addition, define
    \begin{equation*}
        \omega_N(G) := \{(y^{(j_1)}, \ldots, y^{(j_m)}) : (y^{(j_k)}, y^{(j_{k+1})}) \in E \hspace{2mm} \forall 1 \leq k < m, \hspace{2mm} m \leq N\} \cup \{\emptyset\}
    \end{equation*}
    to be the set of all walks of length $\leq N$ on $G$, including the empty walk $\emptyset$. Define similarly
    \begin{equation*}
        \omega(G) := \bigcup_{N = 1}^\infty \omega_N(G)
    \end{equation*}
    to be the set of all walks on $G$.
\end{definition}

\begin{remark} \label{remark:CDE_assumptions}
    In the context of the above construction, $A^i$ can be thought of as the adjacency matrix of the sub-graph of $G$ which corresponds to flows controlled by the component $x^i$. We will make three simplifying assumptions about the CDE \eqref{eq:linear_cde}: \vspace{2mm}
    \begin{enumerate}
        \item $A^i$ are maximally sparse, in the sense that for each $i$, there exist unique $j,k$ such that $A^i_{j,k} \neq 0$. That is, each component $i$ of the path $x$ describes exactly one flow between two nodes $y^{(j)}, y^{(k)}$.
        \item $A^i$ are disjoint, in the sense that $A^{i_1}_{j,k}, A^{i_2}_{j,k} \neq 0$ if and only if $i_1 = i_2$, so that there are no double edges in the graph $G$. That is, the flow from a node $y^{(j)}$ to $y^{(k)}$ is controlled by at most one component of the path $x$.
        \item $A^i \in \{0,1\}^{m \times m}$ are binary matrices.
    \end{enumerate}\vspace{2mm}
    We will, however, allow $G$ to contain loops. We stress that these assumptions are made purely to allow the analysis which follows (Lemma \ref{lemma:graph_sparsity}), and are not in general necessary to construct an Euler scheme for the CDE \eqref{eq:linear_cde}. In particular, Expression \eqref{eq:EulerWalks} below is easily modified to account for the general case.
\end{remark}

\begin{definition}
    Given the construction of Definition \ref{defn:linear_cde} and the assumptions of Remark \ref{remark:CDE_assumptions}, define recursively the map
    \begin{align*}
        \mathcal{F} : \omega(G) &\to \bigcup_{N = 0}^\infty \{0, \ldots, d\}^N \vspace{1mm}\\
        \mathcal{F}(\emptyset) &= \emptyset\\
        \mathcal{F}((y^{(j)}, y^{(j_2)})) &= (i) \quad \text{where } A^i_{j_1,j_2} \neq 0\\
        \mathcal{F}((y^{(j_1)}, \ldots, y^{(j_m)})) &= \mathcal{F}((y^{(j_1)}, \ldots, y^{(j_{m-1})})) * (i) \quad \text{where } A^i_{j_{m-1},j_m} \neq 0
    \end{align*}
    where $*$ denotes the concatenation of multi-indices. For a walk $w \in \omega(G)$, $\mathcal{F}(w)$ then describes the \say{flow} along the walk as a multi-index corresponding to channels of $x$.
\end{definition}

The signature coefficients required to update the value of $y^{(j)}$ using the scheme \eqref{eq:EulerFull} can then be easily read off as the walks of length $\leq N$ ending at $y^{(j)}$. That is, under the assumptions of Remark \ref{remark:CDE_assumptions}, we may write the $N$-step Euler scheme \eqref{eq:EulerFull} as
\begin{equation} \label{eq:EulerWalks}
    y_t \approx y_s + \sum_{\substack{w = \left(y^{(j_1)}, \ldots, y^{(j_m)}\right) \\ \in \omega_{N}(G) \setminus \{\emptyset\}}} S(x)^{\mathcal{F}(w)}_{[s,t]} y_s^{(j_1)} e_{j_m}.
\end{equation}

\begin{remark}
    Given the discussion in Section \ref{section:grid_retrieval}, it suffices to run the computation of coefficients only for the longest walks originating from any given node, since coefficients corresponding to shorter walks can be extracted from the PDE grid. Depending on whether these coefficients are computed serially, one may also choose to cache large chunks of the computation in the case of walks coinciding on long initial segments. For example, having computed the coefficient corresponding to a walk $(y^{(j_1)}, y^{(j_2)}, \ldots, y^{(j_{m-1})}, y^{(j_m)})$, we immediately obtain the coefficient for $(y^{(j_1)}, y^{(j_2)}, \ldots, y^{(j_{m-1})})$, and may reuse a significant portion of the PDE grid to compute $(y^{(j_1)}, y^{(j_1)}, \ldots, y^{(j_{m-1})}, y^{(j_k)})$ for $k \neq m$.
\end{remark}

With the above representation in mind, it is clear that the CDEs which require the fewest signature coefficients are those which limit the number of walks through the corresponding directional graph. In particular, one may choose to look at CDEs whose directional graph representation is acyclic or of low maximum out-degree. However, to properly quantify the benefit of computing isolated coefficients over the naive computation of a full truncated signature, we should introduce the idea of the \textit{sparsity} of an Euler scheme as the ratio of the number of terms in the scheme \eqref{eq:EulerWalks} to the length of the signature up to order $N$. As we will see, under the assumptions of Remark \ref{remark:CDE_assumptions}, the sparsity is mostly independent of the number of edges $d$ in $G$, and indeed the number of walks.

\begin{definition}[Sparsity] \label{def:sparsity}
    Let $\lVert \mathcal{E}_{s,t}^N \rVert_0$ denote the number of terms in the sum \eqref{eq:EulerWalks}. Define the sparsity $s(\mathcal{E}_{s,t}^N)$ of the N-step Euler scheme as
    \begin{equation}
        s(\mathcal{E}_{s,t}^N) = \frac{d-1}{d^{N+1} - 1} \lVert \mathcal{E}_{s,t}^N \rVert_0.
    \end{equation}
\end{definition}

We can immediately obtain a basic bound on the sparsity with respect to the maximum out-degree of $G$, which simplifies in the case of regularity.

\begin{lemma} \label{lemma:graph_sparsity}
    Consider a CDE of the form \eqref{eq:linear_cde} and let $G = (V,E)$ be its corresponding directional graph, with $|V| = m$ and $|E| = d$. Suppose the assumptions of Remark \ref{remark:CDE_assumptions} hold. Let $\Delta^+(G)$ denote the maximum out-degree of $G$. Then
    \begin{equation*}
        s(\mathcal{E}_{s,t}^N) \leq m  \cdot \frac{d-1}{d^{N+1} - 1} \cdot \frac{\left[ \Delta^+(G) \right]^{N+1} - 1}{\Delta^+(G) - 1}
    \end{equation*}
    with equality if and only if $G$ is $\Delta^+(G)$-regular with respect to the out-degree. In the case of equality, when $\Delta^+(G)$ is large, we have
    \begin{equation*}
        s(\mathcal{E}_{s,t}^N) \approx m^{1-N}.
    \end{equation*}
\end{lemma}
\begin{proof}
    It is easy to see that, for any $j$,
    \begin{equation*}
    |\omega_N(G)| \leq m\frac{\left[ \Delta^+(G) \right]^{N+1} - 1}{\Delta^+(G) - 1}
    \end{equation*}
    from which the stated inequality follows by noting that $\left\lVert \mathcal{E}^N_{s,t} \right \rVert_0 = |\omega_N(G)|$. When $G$ is $\Delta^+(G)$-regular, we have $d = |E| = m \cdot \Delta^+(G)$, and so when $\Delta^+(G)$ is large,
    \begin{align*}
        s(\mathcal{E}_{s,t}^N) &\leq m  \cdot \frac{d-1}{d^{N+1} - 1} \cdot \frac{\left[ \Delta^+(G) \right]^{N+1} - 1}{\Delta^+(G) - 1}\\
        &= m  \cdot \frac{(m\cdot\Delta^+(G))-1}{(m\cdot\Delta^+(G))^{N+1} - 1} \cdot \frac{\left[ \Delta^+(G) \right]^{N+1} - 1}{\Delta^+(G) - 1}\\
        &\approx m^{1-N}.
    \end{align*}
\end{proof}

It is notable that the approximate sparsity in the regular case is independent of $d$, although this is not surprising given the correspondence of walks on $G$ with signature coefficients of $x$. We will now consider some concrete examples of graphs constructed on $D$-dimensional lattices, and offer a slight refinement to Lemma \ref{lemma:graph_sparsity} in this case.

\begin{definition}
    Let $D, m \in \mathbb{N}$. For a $D$-dimensional coordinate $i = (i_1, \ldots, i_D) \in \{1, \ldots, m\}^D$, let $i^-$ denote the set of coordinates
    \begin{equation*}
        i^- = \{(i_1, \ldots, i_k - 1, \ldots, i_D) : 1 \leq k \leq D, \hspace{1mm} i_k > 1\}.
    \end{equation*}
    Given a path $\lambda \in C_1([0,1], \mathbb{R}^d)$ where $d = \sum_{i \in \{1,\ldots,m\}^D} |i^-|$, we will refer to
    \begin{equation} \label{eq:birth_only}
        dy^{(i)}_t = \sum_{j \in i^-} y^{(i)}_t d\lambda^{(j)}_t
    \end{equation}
    as the birth-only CDE on a $D$-dimensional lattice. Similarly, given a path $\mu \in C_1\left([0,1], \mathbb{R}^{m^D}\right)$, we will call
    \begin{equation} \label{eq:birth_death}
        dy^{(i)}_t = -y_t^{(i)} d\mu^{(i)}_t + \sum_{j \in I^-} y^{(j)}_t d\lambda^{(j)}_t
    \end{equation}
    the birth-death CDE on a $D$-dimensional lattice.
\end{definition}

\begin{theorem} \label{thm:lattice_sparsity}
    Consider the $N$-step Euler scheme \eqref{eq:EulerFull} applied to the CDE \eqref{eq:birth_only} or \eqref{eq:birth_death}. Then there exists a constant $C$ depending on $D$ and $N$ such that
    \begin{equation*}
        s(\mathcal{E}_{s,t}^N) \sim C m^{D(1-N)}
    \end{equation*}
     as $m \to \infty$. Moreover, $C = N$ if $D = 1$ and
     \begin{equation*}
         C \to \frac{D}{D-1} \quad \text{as } N \to \infty
     \end{equation*}
     if $D > 1$.
\end{theorem}

\begin{proof}
    We will prove the birth-only case, as the birth-death case follows the same argument. Consider the graph representation of the CDE \eqref{eq:birth_only}. It is clear that for any $i = (i_1, \ldots, i_D)$ such that $i_k \leq m - N$ for all $k$, the number of distinct walks of length $\leq N$ starting at node $y^{(i)}$ equals $C_1 := \sum_{k= 1}^N D^k =  (D^{N+1} - D)/(D-1)$ if $D > 1$ and $C_1 := N $ if $D = 1$. The remaining $\mathcal{O}(m^{D-1})$ many nodes have strictly fewer walks. It follows that
    \begin{equation*}
        \lVert \mathcal{E}_{s,t}^N \rVert_0 = |\omega_N(G)| = C_1(m-N)^{D} + \mathcal{O}(m^{D-1}).
    \end{equation*}
    By a similar argument, for any $i = (i_1, \ldots, i_D)$ such that $i_k \leq m - 1$, the node $y^{(i)}$ has out-degree $D$. Thus
    \begin{equation*}
        d = D(m - 1)^D + \mathcal{O}(m^{D-1}).
    \end{equation*}
    The result follows immediately from Definition \ref{def:sparsity}.
\end{proof}

\subsection{A Toy Example}

As an example of a 1-dimensional birth-death CDE, consider tracking the generations in a population of organisms over a year. Let $y^{(k)}_t$ denote the population of the $k^{th}$ generation at time $t$ for $1\leq k \leq n$, and let $y^{(>n)}_t$ denote the cumulative population of all generations beyond the $n^{th}$. We will assume $y^{(k)}_0 = 1$ for all $1 \leq k \leq n$ and $y^{(>n)}_0 = 1$. \vspace{-5mm}

\resizebox{0.95\textwidth}{!}{
\begin{tikzpicture}[line cap=round,line join=round,>=triangle 45,x=1.0cm,y=1.0cm]
\clip(-4,-2) rectangle (12,3.5);

\node (zero) [draw, circle, minimum size=1cm] at (-2,1) {$y^{(1)}$};
\node (one) [draw, circle, minimum size=1cm] at (0.5,1) {$y^{(2)}$};
\node (two) [draw, circle, minimum size=1cm] at (3,1) {$y^{(3)}$};
\node (en) [draw, circle, minimum size=1cm] at (6,1) {$y^{(n)}$};
\node (enplusone) [draw, circle, minimum size=1cm] at (8.5,1) {$y^{(>n)}$};

\draw [->] (zero) -- node [midway, above] {$d\lambda_1(t)$} (one);
\draw [->] (one) -- node [midway, above] {$d\lambda_2(t)$} (two);
\draw [->, very thick, dashed] (two) -- (en);
\draw [->] (en) -- node [midway, above] {$d\lambda_n(t)$} (enplusone);

\draw[->] (zero) edge[in=-120, out=-60, loop] node[below] {$-d\mu_1(t)$} ();
\draw[->] (one) edge[in=-120, out=-60, loop] node[below] {$-d\mu_2(t)$} ();
\draw[->] (two) edge[in=-120, out=-60, loop] node[below] {$-d\mu_2(t)$} ();
\draw[->] (en) edge[in=-120, out=-60, loop] node[below] {$-d\mu_n(t)$} ();
\draw[->] (enplusone) edge[in=-120, out=-60, loop] node[below] {$-d\mu_{>n}(t)$} ();
\end{tikzpicture}
}

Suppose the birth rate is seasonal, such that over a year, the cumulative birth rate of the $k^{th}$ generation follows the sigmoidal function
\begin{equation*}
    \lambda_k(t) = \frac{a(k)}{1+\exp(-b(k)(t - c(k)))}, \quad t \in [0,1],
\end{equation*}

where $a(k)$ can be interpreted as a parameter controlling the total birth within a cycle, $b(k)$ controlling the speed of birth and $c(k)$ the start of the mating season. Moreover, suppose that within the year, there is a death-triggering event, such as an epidemic, where that the cumulative death rate follows the sigmoidal function
\begin{equation*}
    \mu_k(t) = \frac{d(k)}{1+\exp(-f(k)( t - t_{\text{ep}}))}, \quad t \in [0,1],
\end{equation*}

where $d(k)$ is the total impact on generation $k$, $f(k)$ is the duration of the impact on generation $k$ and $t_{\text{ep}}$ is the time of the peak of the epidemic. Figure \ref{fig:epidemic} shows the results of a $5$-step Euler scheme on the corresponding CDE, for some fixed choices of the above parameters.

\section{Conclusion} \label{chapter:conclusion}

In this paper, we have discussed methods for the computation of a single signature coefficient, $S(x)^I$, located at depth $n$ of the signature of a piecewise linear path $x$ composed of $L$-many linear segments. With Chen's relation, this can be achieved in $\mathcal{O}(Ln^2)$ time, and reduced to $\mathcal{O}(Ln)$ time after a parallelising the computation. We proposed the use of signature kernels to construct an approximation to signature coefficients by \say{filtering} out the required coefficient from the transform. The resulting algorithm has a serial complexity of $\mathcal{O}(Ln2^n)$, but with the benefit of being highly parallelisable, resulting in a parallel complexity of $\mathcal{O}(L)$. With the advent of massively parallel processors, this poses a viable alternative to Chen's relation when fast computation times are critical. We showed that, as with Chen's relation, intermediate steps in the computation allow for the extraction of related coefficients which are often required in practical applications. Finally, we presented a simple example use case for sparse collections of signature coefficients by considering $N$-step Euler schemes for sparse CDEs and state-space models.\par\vspace{2cm}

\begin{figure*}[h!]
    \centering
    \includegraphics[width = 0.92\textwidth, trim={0 0.7cm 0 1.8cm},clip]{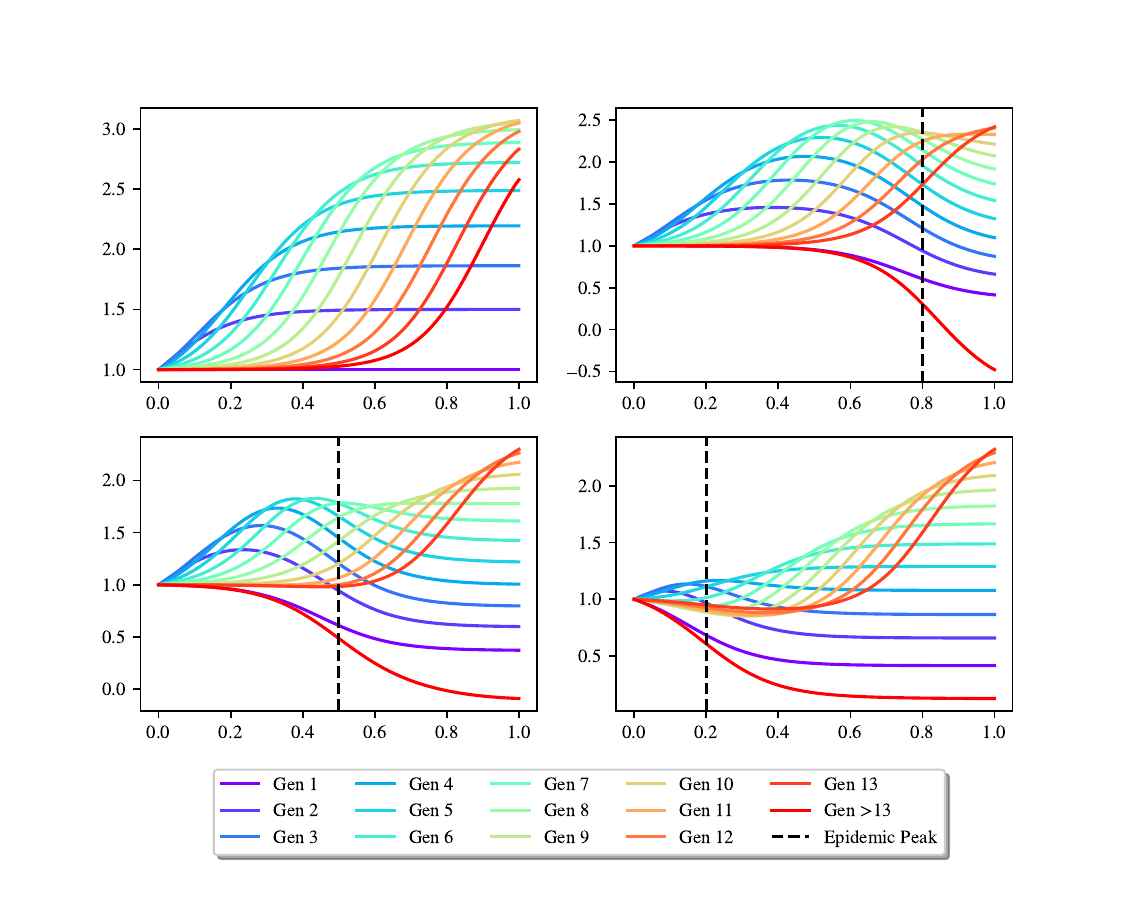}
    \caption{Generational model with $n = 13$ approximated using a $5$-step Euler scheme, leading to a sparsity of  $s(\mathcal{E}_{s,t}^5) = 5 \times 10^{-5}$. Model parameters are chosen as $a(k) = 1$, $b(k) = f(k) = 10$, $c(k) = k / 14$, $d(k) = (14-k)/14$. Epidemic is chosen to never occur (top left), start late (top right), start mid-year (bottom left) or start early (bottom right).}
    \label{fig:epidemic}
\end{figure*}

\newpage
\appendix
\section{Proof of Proposition \ref{thm:kderiv}} 
\label{proof:kderiv}

We will first prove a lemma detailing how the derivative acts on levels of the signature, before moving on to the proof of the theorem itself.

\begin{lemma}
    For any $z \in C_1([0,1], \mathbb{R}^n)$,
    \begin{equation*}
        \frac{\partial ^n}{\partial \lambda_1 \cdots \partial \lambda_n} \biggr\rvert_{\lambda = 0} \hspace{1mm} S(\lambda \odot z)^{(m)}_{[0,1]} = 
        \begin{cases}
            \sum_{J \in \mathcal{P}(1, \ldots, n)} S(z)^{J}_{[0,1]} e_J, & m = n \\
            \hfil \mathbf{0}, & m \neq n,
        \end{cases}
    \end{equation*}
    where $e_J = e_{j_1} \cdots e_{j_n}$ for $J = (j_1, \ldots, j_n)$.
\end{lemma}

\begin{proof}
For $m = n$,
    \begin{align*}
    \frac{\partial ^n}{\partial \lambda_1 \cdots \partial \lambda_n} \biggr\rvert_{\lambda = 0} \hspace{1mm}  S(\lambda \odot z)^{(m)}_{[0,1]} &= \frac{\partial ^n}{\partial \lambda_1 \cdots \partial \lambda_n} \biggr\rvert_{\lambda = 0} \hspace{1mm} \int_{0 < t_1 < \cdots < t_n < 1} d(\lambda \odot z)_{t_1} \otimes \cdots \otimes d(\lambda \odot z)_{t_n} \\
    &= \frac{\partial ^n}{\partial \lambda_1 \cdots \partial \lambda_n} \biggr\rvert_{\lambda = 0} \hspace{1mm} \int_{0 < t_1 < \cdots < t_n < 1} (\lambda \odot dz_{t_1}) \otimes \cdots \otimes (\lambda \odot dz_{t_n}) \\
    &= \sum_{J \in \mathcal{P}(1, \ldots, n)} \int_{0 < t_1 < \cdots < t_k < 1} (e_{j_1} \odot dz_{t_1}) \otimes \cdots \otimes (e_{j_n} \odot dz_{t_n}) \\
    & = \sum_{J \in \mathcal{P}(1, \ldots, n)} S(z)^{J}_{[0,1]} e_J.
    \end{align*}
  
    By a similar calculation, it is easy to see that the derivative is $\mathbf{0} \in V^{\otimes m}$ if $m \neq n$.
\end{proof}

\vspace{5mm}

\newenvironment{proof1}{\paragraph{Proof of Proposition \ref{thm:kderiv}:}}{\hfill$\qed$}

\begin{proof1}
    \begin{align*}
        \frac{\partial ^n}{\partial \lambda_1 \cdots \partial \lambda_n} \biggr\rvert_{\lambda = 0} \hspace{1mm} \mathbf{k}_{x^I,\lambda \odot y} \hspace{0.5mm} &= \frac{\partial ^n}{\partial \lambda_1 \cdots \partial \lambda_n} \biggr\rvert_{\lambda = 0} \hspace{1mm} \left<S(x^I)_{[0,1]}, \hspace{1mm} S(\lambda \odot y)_{[0,1]} \right>\\
        &= \frac{\partial ^n}{\partial \lambda_1 \cdots \partial \lambda_n} \biggr\rvert_{\lambda = 0} \hspace{1mm} \sum_{i=0}^{\infty} \left<S(x^I)^{(i)}_{[0,1]}, \hspace{1mm} S(\lambda \odot y)^{(i)}_{[0,1]} \right>_{V^{\otimes i}}\\
        &= \sum_{i=0}^{\infty} \left<S(x^I)^{(i)}_{[0,1]}, \hspace{1mm} \frac{\partial ^n}{\partial \lambda_1 \cdots \partial \lambda_n} \biggr\rvert_{\lambda = 0} \hspace{1mm} S(\lambda \odot y)^{(i)}_{[0,1]} \right>_{V^{\otimes i}}\\
        &= \left<S(x^I)^{(n)}_{[0,1]}, \hspace{1mm}  \sum_{J \in \mathcal{P}(1, \ldots, n)} S(y)^{J}_{[0,1]} e_J\right>_{V^{\otimes n}} \\
        &= \sum_{J \in \mathcal{P}(1, \ldots, n)} S(x^I)^{J}_{[0,1]} S(y)^{J}_{[0,1]} \\
        &= \frac{1}{n!} S(x)^{\mathcal{P}(I)}_{[0,1]},
    \end{align*}

    where the interchange of summation and differentiation is justified by the uniform convergence of the series of derivatives.

\end{proof1}

\newpage
\section{Proof of Proposition \ref{prop:vandermonde_inversion}} \label{proof:vandermonde_inversion}

\begin{proof}
    We define $\mathbf{e}^{(M+1)}_j(\mathbf{x})$, $\mathbf{e}^{(M+1)}_{j, l}(\mathbf{x})$ and $d_j$ as in \cite{arafat2022fast} and write $V_{M+1, n} = B_{n,M}$, $x_j = \beta_{j-1}$ to match the notation. By \cite[Theorem 2]{arafat2022fast} we have that
    \begin{equation*}
        \alpha_i = (B_{n,M})^{-1}_{i,1} = (V_{M+1, n})^{-1}_{i+1,1} = \frac{(-1)^M \mathbf{e}^{(M+1)} _{M+1, i+1} (\mathbf{x})}{d_{i+1}}.
    \end{equation*}
    By \cite[Equation 6]{arafat2022fast}, we have
    \begin{align*}
        \mathbf{e}^{(M+1)} _{M+1, i+1} (\mathbf{x}) &= \frac{\partial}{\partial x_{i+1}} \mathbf{e}^{(M+1)}_{M+1} (\mathbf{x})\\
        &= \frac{\partial}{\partial x_{i+1}} \prod_{j=1}^{M+1} x_j\\
        &= \prod_{\substack{j=1 \\ j \neq {i+1}}}^{M+1} x_j.
    \end{align*}
    Substituting in $d_{i+1} = x_{i+1}^n \prod\limits_{\substack{j=1 \\ j\neq i+1}}^{M+1} (x_{i+1} - x_j)$ and $x_j = \beta_{j-1}$ gives the result.
\end{proof}
\vspace{5mm}

\section{Proof of Proposition \ref{prop:vandermonde}} \label{proof:vandermonde}

\begin{proof} 
    Suppose $D^{(n)}_\lambda \cdot \big\rvert_{\lambda = 0}$ takes the form
    \begin{equation*}
        D^{(n)}_\lambda f(\lambda)\Big\rvert_{\lambda = 0} = \sum_{\widetilde \lambda \in \Lambda} C_{\widetilde \lambda} f(\widetilde \lambda)
    \end{equation*}
    for some constants $C_{\widetilde \lambda} \in \mathbb{R}$. Define
    \begin{align*}
        \ell_{\text{max}} &:= \max_{\widetilde \lambda \in \Lambda} \lVert \widetilde \lambda \odot y \rVert_1,\\
        C_{\text{max}} &:= \max_{\widetilde \lambda \in \Lambda} |C_{\widetilde \lambda}|.
    \end{align*}
    Then by Proposition \ref{prop:truncatedkernel},
    \begin{align*}
        &\left\lVert n! \sum_{i=0}^M \alpha_i D^{(n)}_\lambda \mathbf{k}^{\lambda \mhyphen \beta_i}_{x^I, y} \Big\rvert_{\lambda = 0} - S(x)^{\mathcal{P}(I)}\right \rVert\\
        &= \left \lVert n! \sum_{i=0}^M \alpha_i D^{(n)}_\lambda \sum_{j=n+M+1}^\infty \left< S\left(x^I\right)^{(j)}, S( \beta_i \lambda \odot y)^{(j)} \right>_{V^{\otimes j}}\Bigg\rvert_{\lambda = 0} \right \rVert \tag{Prop. \ref{prop:truncatedkernel}} \\
        &\leq n! \cdot C_{\text{max}} \cdot |\Lambda| \sum_{i=0}^M |\alpha_i| \sum_{j=n+M+1}^\infty \frac{\beta_i^j \left \lVert x^I \right \rVert_1^j \ell_{\text{max}}^j }{(j!)^2}  \tag{Cor. \ref{cor:cauchyfactorialdecay}}\\
        &\leq n! \cdot C_{\text{max}} \cdot |\Lambda| (M+1) \max_{0\leq i \leq M}{|\alpha_i|} \sum_{j=n+M+1}^\infty \frac{\left \lVert x^I \right \rVert^j_1 \ell_{\text{max}}^j}{(j!)^2}\\
        &\leq n! \cdot C_{\text{max}} \cdot |\Lambda| (M+1) \max_{0\leq i \leq M}{|\alpha_i|} \frac{(\left \lVert x^I \right \rVert_1 \ell_{\text{max}})^{n+M}}{((n+M)!)^2} \sum_{j=1}^\infty \frac{((n+M)!)^2}{((n+M+j)!)^2} \left \lVert x^I \right \rVert^j_1 \ell_{\text{max}}^j\\
        &\to 0 \tag{Cond. \eqref{eq:alphacond}}.
    \end{align*}
\end{proof}

\newpage
\section{Proof of Proposition \ref{prop:M_bound}} \label{proof:M_bound}

\begin{proof}
By the proof of Proposition \ref{prop:vandermonde}, we have

\begin{align*}
    &\left\lVert n! \sum_{i=0}^M \alpha_i D^{(n)}_\lambda \mathbf{k}^{\lambda \mhyphen \beta_i}_{x^I, y} \Big\rvert_{\lambda = 0} - S(x)^{\mathcal{P}(I)}\right \rVert\\
    &\leq n! \cdot C_{\text{max}} \cdot |\Lambda| (M+1) \max_{0\leq i \leq M}{|\alpha_i|} \frac{(\left \lVert x^I \right \rVert_1 \ell_{\text{max}})^{n+M}}{((n+M)!)^2} \sum_{j=1}^\infty \frac{((n+M)!)^2}{((n+M+j)!)^2} \left \lVert x^I \right \rVert^j_1 \ell_{\text{max}}^j.
\end{align*}

Moreover,
\begin{align*}
    \sum_{j=1}^\infty \frac{((n+M)!)^2}{((n+M+j)!)^2}\left \lVert x^I \right \rVert^j_1 \ell_{\text{max}}^j &\leq \sum_{j=1}^\infty \frac{\left \lVert x^I \right \rVert^j_1 \ell_{\text{max}}^j}{(j!)^2}\\
    &= I_0\left(2 \sqrt{\left \lVert x^I \right \rVert_1 \ell_{\text{max}}}\right) - 1.
\end{align*}

where $I_0$ is the modified Bessel function of the first kind, and
\begin{align*}
    \max_{0\leq i \leq M}{|\alpha_i|} &= \max_{0\leq i \leq M} \frac{1}{\beta_i^n} \prod_{\substack{j=0 \\ j\neq i}}^M \frac{\beta_j}{|\beta_i - \beta_j|} \\
    &= \max_{0\leq i \leq M}{\frac{M+1}{i+1} \prod^M_{\substack{j = 0 \\ j \neq i}} \left\lvert 1 - \left(\frac{j + 1}{i + 1}\right)^{1/n} \right\rvert^{-1}}\\
    &\leq \max_{0\leq i \leq M}{\frac{M+1}{i+1} \left(\left(\frac{i+2}{i+1}\right)^{1/n} - 1\right)^{-M}}\\
    &\leq (M+1) \left(2^{1/n} - 1\right)^{-M}.
\end{align*}

Setting
\begin{align*}
    A &:= n! \cdot C_{\text{max}} \cdot |\Lambda| \left(\left \lVert x^I \right \rVert_1 \ell_{\text{max}}\right)^n \left[I_0\left(2 \sqrt{\left \lVert x^I \right \rVert_1 \ell_{\text{max}}}\right) - 1\right]\\
    B &:= \ell_{\text{max}} \left(2^{1/n} - 1\right)^{-1}
\end{align*}

gives
\begin{equation*}
    \left\lVert n! \sum_{i=0}^M \alpha_i D^{(n)}_\lambda \mathbf{k}^{\lambda \mhyphen \beta_i}_{x^I, y} \Big\rvert_{\lambda = 0} - S(x)^{\mathcal{P}(I)}\right \rVert \leq A (M+1)^2 B^M \frac{ \left\lVert x^I \right \rVert_1^{n+M}}{((n+M)!)^2}
\end{equation*}

\end{proof}

\newpage
\section{Proof of Theorem \ref{thm:semiordered}} \label{proof:semiordered}

\begin{proof}
    By Chen's relation (Proposition \ref{prop:chen}) and Proposition \ref{prop:linearsignature}, it is easy to see that
    \begin{equation*}
        S(z)^{J} = \frac{1}{l_1! \cdots l_m!} \mathbf{1}\{J \in \mathcal{P}(I_1) * \cdots * \mathcal{P}(I_m)\}
    \end{equation*}
    for all $J = (j_1, \ldots, j_n)$. By the construction, we have
    \begin{align*}
        &\sum_{i=0}^M \alpha_i D^{(n)}_\lambda \sum_{j=0}^{n+M} \left<S(\beta_i \lambda \odot x^I)^{(j)}, S(z)^{(j)}\right>_{V^{\otimes j}} \Bigg\rvert_{\lambda = 0}\\
        &= \sum_{J \in \mathcal{P}(1, \ldots, n)} S(x^I)^J S(z)^J\\
        &=\frac{1}{l_1! \cdots l_m!} S(x)^{\mathcal{P}(I_1) * \cdots * \mathcal{P}(I_m)}.
    \end{align*}
    Moreover, following the same argument as in the proof of Proposition \ref{prop:vandermonde} in Appendix \ref{proof:vandermonde} but with $y$ replaced by $z$, the residual terms satisfy
    \begin{align*}
        \left \lVert \sum_{i=0}^M \alpha_i D^{(n)}_\lambda \sum_{j=n+M+1}^\infty \left< S(\beta_i \lambda \odot x^I)^{(j)}, S(z)^{(j)} \right>_{V^{\otimes j}} \Bigg\rvert_{\lambda = 0} \right \rVert \to 0
    \end{align*}
    as $M \to \infty$, for any $\beta_i$ satisfying Condition \eqref{eq:alphacond}. It follows that
    \begin{equation*}
        \left\lVert \sum_{i=0}^M \alpha_i D^{(n)}_\lambda \mathbf{k}^{\lambda \mhyphen \beta_i}_{x^I, \hspace{0.5mm} z} \Big\rvert_{\lambda = 0} - \frac{1}{l_1! \cdots l_m!} S(x)^{\mathcal{P}(I_1) * \cdots * \mathcal{P}(I_m)} \right\rVert \to 0.
    \end{equation*}
\end{proof}

\bibliographystyle{siamplain}
\bibliography{references}
\end{document}